\documentclass[10pt]{article}

\usepackage{amsfonts, amsmath, amssymb,latexsym, bbm, amsthm}
\usepackage{mathtools}

\usepackage[letterpaper, margin=1in]{geometry}

\usepackage{tikz}
\usetikzlibrary{positioning}
\usepackage{hyperref}
\usepackage{multirow}
\usepackage{mathtools}

\newcommand{\Stab}{\operatorname{Stab}}
\newcommand{\Aut}{\operatorname{Aut}}
\newcommand{\Cl}{\operatorname{Cl}}
\newcommand{\Id}{\operatorname{Id}}
\newcommand{\Gal}{\operatorname{Gal}}
\newcommand{\Disc}{\operatorname{Disc}}
\newcommand{\Det}{\operatorname{Det}}
\newcommand{\GL}{\operatorname{GL}}
\newcommand{\Frac}{\operatorname{Frac}}
\newcommand{\Ind}{\operatorname{Ind}}

\newcommand{\SL}{\operatorname{SL}}
\newcommand{\Sym}{\operatorname{Sym}}
\newcommand{\SO}{\operatorname{SO}}

\newcommand{\Vol}{\operatorname{Vol}}
\newcommand{\Res}{\operatorname{Res}}

\newcommand{\Nm}{\operatorname{Nm}}
\newcommand{\rk}{\operatorname{rk}}
\newcommand{\Cond}{\operatorname{Cond}}
\makeatletter
\newcommand{\vast}{\bBigg@{4}}
\newcommand{\Vast}{\bBigg@{5}}
\makeatother

\newcommand{\bOne}{\mathbbm{1}}
\newcommand{\C}{\mathbb{C}}
\newcommand{\bC}{\mathbb{C}}
\newcommand{\CC}{{\rm C}}
\newcommand{\gen}{{\rm gen}}
\newcommand{\bF}{\mathbb{F}}
\newcommand{\Ss}{\mathcal{S}}
\newcommand{\bP}{\mathbb{P}}
\newcommand{\Q}{\mathbb{Q}}
\newcommand{\bQ}{\mathbb{Q}}
\newcommand{\R}{\mathbb{R}}
\newcommand{\bR}{\mathbb{R}}
\newcommand{\Z}{\mathbb{Z}}
\newcommand{\bZ}{\mathbb{Z}}
\newcommand{\disc}{\text{Disc}}

\newcommand{\W}{\mathcal{W}}

\newcommand{\LL}{\mathcal{L}}
\newcommand{\KK}{\mathcal{K}}
\newcommand{\OO}{\mathcal{O}}
\newcommand{\cO}{\mathcal{O}}
\newcommand{\FF}{\mathcal{F}}

\newcommand{\VV}{\mathcal{V}}
\newcommand{\F}{\mathbb{F}}

\newcommand{\fa}{\mathfrak{a}}
\newcommand{\fc}{\mathfrak{c}}
\newcommand{\ff}{\mathfrak{f}}

\newcommand{\eps}{\varepsilon}

\def\red{{\rm red}}

\newtheorem{lem}{Lemma}[section]
\newtheorem{thm}[lem]{Theorem}
\newtheorem{cor}[lem]{Corollary}
\newtheorem{prop}[lem]{Proposition}
\newtheorem{conj}[lem]{Assumption}
\newtheorem{lemma}[lem]{Lemma}
\newtheorem{theorem}{Theorem}

\theoremstyle{definition}
\newtheorem{defn}[lem]{Definition}
\newtheorem{rmk}[lem]{Remark}

\newcommand{\NumD}{N_{\rm D_4}}
\newcommand{\NumC}{N_{\CC}}
\newcommand{\NumDD}{N_{\Disc}}

\newcommand{\NNumQ}{\NN_{\q}}
\newcommand{\outeraut}{\phi}
\newcommand{\rotation}{\sigma}
\newcommand{\reflection}{\tau}

\newcommand{\complex}{r_2}

\newcommand{\ci}{J}

\newcommand{\D}{{\rm d}}

\newcommand{\q}{{\rm q}}
\newcommand{\sigmatwo}{\varsigma}
\newcommand{\sigmafour}{\varsigma}
\newcommand{\maximalatp}{M_p}
\newcommand{\Maximalatp}{\mathcal{M}_p}
\newcommand{\resolventmap}{\varphi}

\newcommand{\fundset}{\mathcal{R}^{(\varsigma)}}
\newcommand{\NN}{{\mathcal{N}}}

\newcommand{\unramifiedatp}{\mathcal{U}_p}
\newcommand{\defeq}{\coloneqq}
\newcommand{\Sigmaall}{\Sigma^{{\rm full}}}
\title{The number of quartic $D_4$-fields ordered by conductor}

\author{S.~Ali Altu\u{g}, Arul Shankar, Ila Varma, and Kevin H.~Wilson}

\begin{document}

\maketitle
\begin{abstract}We consider families of number fields of degree 4 whose normal closures over $\bQ$ have Galois group isomorphic to $D_4$, the symmetries of a square. To any such field $L$, one can associate the Artin conductor of the corresponding 2-dimensional irreducible Galois representation with image $D_4$. We determine the asymptotic number of such quartic $D_4$-fields ordered by conductor, and compute the leading term explicitly as a mass formula, verifying heuristics of Kedlaya and Wood. Additionally, we are able to impose any local splitting conditions at any finite number of primes (sometimes, at an infinite number of primes), and as a consequence, we also compute the asymptotic number of order 4 elements in class groups and narrow class groups of quadratic fields ordered by discriminant. 

Traditionally, there have been two approaches to counting quartic fields, using arithmetic invariant theory in combination with geometry-of-number techniques, and applying Kummer theory together with L-function methods. Both of these strategies fall short in the case of $D_4$-fields ordered by conductor since counting quartic fields containing a quadratic subfield with large discriminant is difficult. However, when ordering by conductor, we utilize additional algebraic structure arising from the outer automorphism of $D_4$ combined with both approaches mentioned above to obtain exact asymptotics. \end{abstract}
\section{Introduction}
The main purpose of this article is to determine the asymptotic number
of quartic dihedral fields with bounded {conductor}. If $L$ denotes a
quartic field whose normal closure $M$ over $\bQ$ has Galois group
$\Gal(M/\bQ)$ isomorphic to the group of symmetries of a square, we
refer to $L$ as a $D_4$-field. Furthermore, there is a unique (up to conjugacy)
irreducible $2$-dimensional Galois representation $$\rho_M:
\Gal(\overline{\bQ}/\bQ) \rightarrow \GL_2(\mathbb{C})$$ that factors
through $\Gal(M/\bQ) \cong D_4$. We define the {\em conductor} of $L$ to
be equal to the Artin conductor of $\rho_M$ (see
\cite[Pg.~158-159]{CasselsFrohlich}).
\begin{theorem}\label{MAINTHEOREM}
Let $\NumD^{(r_2)}(X)$ denote the number of isomorphism classes of $D_4$-quartic fields with $4-2r_2$ real embeddings and  conductor bounded by $X$. Then
	$$\NumD^{(0)}(X) \ = \ \frac{1}{4}\cdot \prod_p \left(1 - \frac{1}{p^2} - \frac{2}{p^3} + \frac{2}{p^4}\right)\cdot X \log X + O(X\log\log X);$$
	$$\NumD^{(1)}(X)\ = \ \frac{3}{8} \cdot \prod_p \left(1 - \frac{1}{p^2} - \frac{2}{p^3} + \frac{2}{p^4}\right)\cdot X \log X + O(X\log \log X);$$ 
	$$\NumD^{(2)}(X) \ = \ \frac{1}{8} \cdot \prod_p \left(1 - \frac{1}{p^2} - \frac{2}{p^3} + \frac{2}{p^4}\right)\cdot X \log X + O(X\log \log X).$$ 
\end{theorem}
Understanding the distribution of number fields with fixed signature
and Galois group is a fundamental question in number theory with
several significant applications. For example, the inverse Galois
problem follows from understanding the main terms for the asymptotic
number of field extensions of each fixed degree and Galois closure
over a given base field. Furthermore, if the results are refined
enough to determine the asymptotic number of field
extensions satisfying certain local specifications, then another application of counting number fields is towards understanding
the distribution of torsion in class groups of number fields of fixed
degree, i.e., to proving cases of the Cohen-Lenstra heuristics
\cite{CL} as well as the extensions given by Gerth \cite{gerth},
Cohen-Martinet \cite{CM}, and Malle \cite{malle2010}.  
\smallskip

There are heuristics (see Conjecture 1.2 of \cite{cohenicm})  for the
order of growth for the number of field extensions of each
fixed degree and Galois closure over a given base field when the
extensions are bounded by their (norms of the relative) discriminants,
due to Linnik, Malle, and Turkelli. Linnik
predicted that the number $N_{S_n}(X)$ of $S_n$-number fields of degree
$n$ with discriminant bounded by $X$ satisfies
	$N_{S_n}(X) \sim c_nX,$ for some constant $c_n$ as $X
\rightarrow \infty$. Additionally, the heuristics of Malle \cite{malle2004} imply that the proportion
of degree-$n$ fields with Galois closure $S_n$ amongst all degree-$n$
fields is expected to be 100\% only when $n$ is a prime. Cohen-Diaz y Diaz-Olivier \cite{CDOQuartic} verified a case of Malle's full conjecture in the quartic dihedral case by 
proving that the number of $D_4$-fields with discriminant bounded by $X$
is asymptotically equal to $cX$, where $c \approx .052326$.
\smallskip

 Results of Gauss, Davenport-Heilbronn \cite{dh} and
 Bhargava \cite{BhargavaQuarticCount,BhargavaQuinticCount} established that
 $N_{S_n}(X) \sim c_nX$ for $n = 2$, $3$, $4$, and $5$, where $c_n$ is
 determined to be given by a {\em mass formula}, i.e., the constants $c_n$ take the form of an Euler
 product of local masses. In \cite{bhargavamass}, Bhargava predicted the
 constants $c_n$ for all $n$, explicitly describing them in terms of Euler products of local masses derived from the heuristic assumption that the completions of $S_n$-number fields at different places behave independently of one another. However, if one computes the analogous product of local masses for $D_4$-fields ordered by discriminant, the resulting constant differs in magnitude from that obtained in the asymptotics of Cohen-Diaz y Diaz-Olivier! 
   \smallskip
    
In particular, Corollary 1.4 of
 \cite{CDOQuartic} implies that the number of totally real $D_4$-fields
 with absolute discriminant bounded by $X$ is asymptotically equal to $cX$, where 
	\begin{equation}\label{cdo}
	c = \frac{3}{\pi^2} \cdot\Biggl( \sum_{\substack{[K:\mathbb{Q}]=2\\0 < \Disc(K) < \infty}} \frac{1}{\Disc(K)^2}\cdot\frac{L(1,K/\bQ)}{L(2,K/\bQ)} \Biggr),
	\end{equation}
but whether this constant $c$ can be determined by a mass formula was not addressed.	Our results imply that the number of totally real $D_4$-fields with
        conductor bounded by $X$ is asymptotically equal to a similar sum:
	\begin{equation}\label{lfunctionthm} \NumD^{(0)}(X) \sim \frac{3}{\pi^2} \cdot \Biggl( \sum_{\substack{[K:\mathbb{Q}]=2\\0 < \Disc(K) \leq X}} \frac{1}{\Disc(K)}\cdot\frac{L(1,K/\bQ)}{L(2,K/\bQ)} \Biggr) \cdot X.\end{equation}
Note that the discriminant of a $D_4$-field is equal in magnitude to the product of
its conductor and the discriminant of its quadratic subfield, and the
extra power of $\Disc(K)$ in \eqref{cdo} can essentially be attributed to this
difference in the choice of invariants. Additionally, we are able to prove that the right hand side of \eqref{lfunctionthm} does indeed satisfy a mass formula:
\begin{theorem}\label{2} We have the following:
  \begin{equation*}
\begin{array}{rccl}
  \displaystyle\sum_{\substack{[K:\mathbb{Q}]=2\\0 < \Disc(K) \leq X}}
  &\displaystyle\frac{1}{\Disc(K)}\cdot\frac{L(1,K/\bQ)}{L(2,K/\bQ)} &\sim&
  \displaystyle\frac{\zeta(2)}{2} \cdot\prod_p\left( 1 - \frac{1}{p^2} -
  \frac{2}{p^3} + \frac{2}{p^4}\right) \cdot\log(X);\\[.2in]
\displaystyle  \sum_{\substack{[K:\mathbb{Q}]=2\\-X \leq \Disc(K) < 0}} 
&\displaystyle\frac{1}{|\Disc(K)|}\cdot\frac{L(1,K/\bQ)}{L(2,K/\bQ)} &\sim& \displaystyle\frac{\zeta(2)}{2} \cdot\prod_p\left( 1 - \frac{1}{p^2} - \frac{2}{p^3} + \frac{2}{p^4}\right) \cdot\log(X).
\end{array}
\end{equation*}
\end{theorem}
The existence of mass formulae when ordering by invariants other than the discriminant has been predicted by heuristics of  Kedlaya \cite{kedlayamass} and Wood \cite{melaniemass}. Theorem \ref{MAINTHEOREM} verifies these conjectures in the quartic dihedral case when ordering by conductor. In particular, we show that Question 5.1 of \cite{melaniemass} is answered in the affirmative for $D_4$-fields ordered by conductor (see Equation \ref{eqCXY}), but it is not true for $D_4$-fields ordered by discriminant (see Equation \ref{eqDXY}).
Thus, the choice of invariant proves to be a subtle issue when determining the asymptotics for families of number fields whose Galois closures have a fixed Galois group other than $S_n$. 

\smallskip

Additionally, we are able to determine refined asymptotics for families
of $D_4$-fields with certain prescribed local specifications, but to describe these results, we must first introduce some notation. We say that
$\Sigma=(\Sigma_v)_v$ is a {\it collection of local specifications},
if for each place $v$ of $\Q$, $\Sigma_v$ contains pairs $(L_v,K_v)$ consisting of an \'etale
algebra $L_v$ of $\bQ_v$ of degree 4 along with a quadratic subalgebra $K_v$.  We
say that such a collection $\Sigma$ is {\it acceptable} if for all but
finitely many primes $p$, the set $\Sigma_p$ contains all pairs $(L_p,K_p)$ with conductor indivisible by
$p^2$. Here, the {\em conductor}
$\CC$ of such a pair is equal to $$\CC(L_p,K_p) :=
\Disc(L_p)/\Disc(K_p),$$ and we also let $\CC_p$ denote the $p$-part of
$\CC$. If $\LL(\Sigma)$ denotes all $D_4$-fields $L$ such that $L
\otimes \bQ_v\in\Sigma_v$ for all $v$, and $\NumD(\Sigma,X)$ denotes
the number of isomorphism classes of $D_4$-fields in $\LL(\Sigma)$
whose conductor is bounded by $X$, we then have:
\begin{theorem}\label{congruence conditions}
If $\Sigma = (\Sigma_v)_v$ is an acceptable collection of local
specifications such that $\Sigma_2$ contains every degree 4 \'etale algebra of $\bQ_2$ containing a quadratic subalgebra, then
\begin{equation*}
  \NumD(\Sigma,X) \sim \frac{1}{2}\cdot 
  \biggl(\sum_{(L,K) \in \Sigma_\infty}\frac{1}{\#\Aut(L,K)}\biggr)\cdot 
  \prod_p \biggl(\sum_{(L_p,K_p) \in \Sigma_p}\frac{1}{\#\Aut(L_p,K_p)}\cdot\frac{1}{\CC_p(L_p,K_p)}\biggr)
  \biggl(1-\frac1{p}\biggr)^2\cdot X\log(X),
\end{equation*}
where for all $v$, $\Aut(L_v,K_v)$ consists of the automorphisms of $L_v$ which send $K_v$ to itself.
\end{theorem}

The constant
determined in Theorem~\ref{congruence conditions} is completely
analogous to the constants $c_{n,\Sigma}$ predicted in Equation 4.2 of
\cite{bhargavamass} for the number of $S_n$-fields of degree $n$
satisfying a collection of local specifications with bounded
discriminant. The extra factor of $\frac{p-1}{p}$ that occurs here can
be explained by a double pole for the relevant counting functions,
which also accounts for the factor of $\log(X)$. It is also noteworthy that Theorem \ref{congruence conditions} establishes Sato-Tate equidistribution for the family of $D_4$-fields ordered by conductor, in the sense of \cite{sst}. \\\vspace{-10pt}

Theorem \ref{congruence conditions} allows us to compute the asymptotic number of order 4 elements in
class groups and narrow class groups of quadratic fields ordered by
discriminant. Such elements in the class groups of a quadratic field
$K$ determine $D_4$-fields $L$ whose normal closures over $\bQ$
contain $K$ as the fixed field of $C_4 \subset D_4$. We obtain the following theorem by determining asymptotics for the acceptable collection of $D_4$-fields that arise
in this manner, even when we restrict the set of quadratic fields by imposing local specifications at a finite set of primes. We remark that it is
crucial for the below result that we order $D_4$-fields by conductor  and furthermore, that we can impose acceptable local specifications at an infinite number of primes.
\begin{theorem}\label{classgroups}
For a quadratic field $K$, let $\Cl_{2^k}(K)$
$($resp.~$\Cl_{2^k}^+(K))$ denote the $2^k$-torsion subgroup in its
ideal class group $\Cl(K)$ $($resp.~narrow class group $\Cl^+(K))$. Let $\KK$ denote a family of quadratic fields with prescribed
splitting types at a finite set $S$ consisting of odd primes. We then have:
\begin{itemize}
\item[{\rm (a)}]
$\displaystyle\sum_{\substack{K\in\KK\\0 < \Disc(K) \leq X}} \
(\#\Cl_4(K)-\#\Cl_2(K))\sim
\displaystyle\frac{1}{16} \cdot\prod_{p\in S}m_{\Cl}(p)\cdot
\prod_p\Bigl(1 + \frac{2}{p}\Bigr)\Bigl( 1 - \frac{1}{p}\Bigr)^2\cdot X\log(X),$
\item[{\rm (b)}]
$\displaystyle\sum_{\substack{K\in\KK\\-X \leq \Disc(K) < 0}}
(\#\Cl_4(K)-\#\Cl_2(K))\sim
\displaystyle  \frac{1}{4} \cdot\prod_{p\in S}m_{\Cl}(p)\cdot
\prod_p\Bigl(1 + \frac{2}{p}\Bigr)\Bigl( 1 - \frac{1}{p}\Bigr)^2\cdot X\log(X), \mbox{ and}$
\item[{\rm(c)}]
$\displaystyle\sum_{\substack{K\in\KK\\0 < \Disc(K) \leq X}} \ 
(\#\Cl^+_4(K)-\#\Cl_2(K))\sim
\displaystyle\frac{1}{8} \cdot\prod_{p\in S}m_{\Cl}(p)\cdot
\prod_p\Bigl(1 + \frac{2}{p}\Bigr)\Bigl( 1 - \frac{1}{p}\Bigr)^2\cdot X\log(X).$
\end{itemize}
Here, $m_{\Cl}(p)$ is determined in terms of the prescribed splitting type for $p \in S$:
\begin{equation*}
  m_{\Cl}(p):=
  \displaystyle\frac{2}{p+2} \quad \mbox{ if $p$ ramifies, and }
    \quad m_{\Cl}(p):= \displaystyle\frac{p}{2p+2} \quad
    \mbox{ otherwise}.
\end{equation*}
\end{theorem}

The above result is a generalization of work of Fouvry-Kl\"uners \cite{FKweights} that is derived from their own previous results \cite{fouvrykluners} completely verifying Gerth's extension \cite{gerth} of the Cohen-Lenstra heuristics to the 4-rank of the narrow class group of quadratic fields. In \cite{fouvrykluners}, Fouvry-Kl\"uners compute all moments for the $4$-ranks of narrow class groups of quadratic fields ordered by discriminant.
In conjuction with those results, Theorem \ref{classgroups} gives evidence towards
the belief that the $4$-ranks and the sizes of $2$-torsion subgroups in
class groups and narrow class groups of quadratic fields behave
independently (see Remark {\ref{FK}}).
\smallskip

As a byproduct of the methods
 used to obtain \eqref{lfunctionthm}, we also prove a refinement of Theorem
 \ref{2} that allows for imposing local specifications at a finite
 number of primes.

\begin{theorem}\label{stablefamilies}
Let $\KK$ denote a set of quadratic fields with prescribed splitting types at a finite set $S$ of odd primes. We then have:
\begin{itemize}
\item[{\rm (a)}]  $ \ \displaystyle\sum_{\substack{K\in\KK\\0 < \Disc(K) \leq X}} \ 
  \frac{1}{\Disc(K)}\cdot\frac{L(1,K/\bQ)}{L(2,K/\bQ)} \  \sim 
  \displaystyle\frac{\zeta(2)}{2} \cdot\prod_{p\in S} \frac{m(p)}{2p^2+4p+4}\cdot\prod_p\left( 1 - \frac{1}{p^2} -
  \frac{2}{p^3} + \frac{2}{p^4}\right) \cdot\log(X),$
\item[{\rm (b)}]
  $\displaystyle\sum_{\substack{K\in\KK\\ -X < \Disc(K) \leq 0}}
  \frac{1}{|\Disc(K)|}\cdot\frac{L(1,K/\bQ)}{L(2,K/\bQ)} \sim
  \displaystyle\frac{\zeta(2)}{2} \cdot\prod_{p\in S} \frac{m(p)}{2p^2+4p+4}\cdot\prod_p\left( 1 - \frac{1}{p^2} -
  \frac{2}{p^3} + \frac{2}{p^4}\right) \cdot\log(X),$
\end{itemize}
  where $m(p)$ is determined in terms of the prescribed splitting type for $p \in S$:
\begin{equation*}
  m(p):=\left\{
  \begin{array}{ll}
    p^2+2p+1&
    \mbox{ if $p$ splits;}\\
    p^2+1&
    \mbox{ if $p$ is inert;}\\
    2(p+1)&\mbox{ if $p$ is ramified.}
  \end{array}\right.
\end{equation*}
\end{theorem}

It is interesting to note that the proof of Theorem \ref{2} is entirely contained within Section 5 while the proof of Theorem \ref{stablefamilies} requires the whole article. We first summarize the argument for Theorem \ref{2}. First, for a quadratic field $K$, $\frac{L(1,K/\bQ)}{L(2,K/\bQ)}$ can be written as a product of infinite sums simply using M$\ddot{\rm o}$bius inversion:
	\begin{equation}\label{Lvalues}\frac{L(1,K/\bQ)}{L(2,K/\bQ)} = \Bigl(\sum_{n=1}^\infty \frac{\chi_K(n)}{n}\Bigr)\cdot\Bigl(\sum_{m=1}^\infty \mu(m)\cdot \frac{\chi_K(m)}{m^{2}}\Bigr),\end{equation}
where $\chi_K$ is the quadratic character associated to $K$. The proof then relies on the following observation: weighting this product by $\Disc(K)^{-1}$, the main contribution when summing over quadratic fields $K$ with bounded discriminant as in \eqref{lfunctionthm} comes from certain diagonal terms of the right hand side of \eqref{Lvalues}, i.e., terms where $mn$ is a square. For each $K$, the sum of these diagonal terms is expressible as an Euler product (see Equation \ref{identity}), which then yields Theorem \ref{2}. If we were to instead weight the product in \eqref{Lvalues} by $\Disc(K)^{-2}$ when summing over all quadratic fields $K$ as in \eqref{cdo}, then the non-diagonal terms are no longer negligible. It is these terms that cause the constant $c$ to not follow the heuristics of \cite{bhargavamass,kedlayamass,melaniemass}. We go further to show that the sum over quadratic fields with discriminant bounded by $X$ of the \emph{diagonal terms} in the right hand side of \eqref{Lvalues} weighted by $\Disc(K)^{-2}$ is asymptotically equal to 
	$$\frac{3\cdot11^2}{2^6\cdot 17}\cdot \prod_p\left(1+\frac{1}{p^2}-\frac{1}{p^3}-\frac{1}{p^4}\right) \cdot X,$$ 
	which is equal to the mass formula predicted for the number of $D_4$-fields of bounded discriminant (see Theorem \ref{5.4}).

\smallskip

 Before describing the proof of Theorem \ref{MAINTHEOREM}, we compare our argument for obtaining \eqref{lfunctionthm} to that of Cohen-Diaz y Diaz-Olivier implying \eqref{cdo}.  Because of the convergence of $\sum_{K}\Disc(K)^{-2}$ when $K$ runs over all quadratic fields, the main contribution to the sum in \eqref{cdo} comes from the terms indexed by $K$ with $\Disc(K) < X^{1/2}$, i.e., quadratic fields with small discriminant relative to $X$. However, this is not the case for the sum in \eqref{lfunctionthm}; furthermore, the combination of Kummer theory and $L$-function methods utilized in \cite{CDOQuartic} to prove \eqref{cdo} can only be used to determine the asymptotic number of $D_4$-fields with conductor bounded by X whose quadratic subfield has small discriminant relative to $X$ (see Theorem \ref{thm:analyticmain2}). Additionally, adapting the geometry-of-numbers techniques of \cite{BhargavaQuarticCount} in combination with Wood's parametrization of quartic rings with a quadratic subring \cite{WoodThesis} is also limited to counting $D_4$-fields whose quadratic subfields have small discriminant. 
 
Nevertheless, we obtain Theorem \ref{MAINTHEOREM} and subsequently \eqref{lfunctionthm} by employing algebraic properties of the conductor $\CC(L)$ of a $D_4$-field $L$, namely that it is invariant under the outer automorphism $\phi$ of $D_4$. More precisely, $\phi$ acts on the Galois group of the normal closure $M$ over $\bQ$ of $L$, and the fixed field of $\phi(\Gal(M/L))$ is another $D_4$-field $\phi(L)$ in $M$ which is not isomorphic to $L$. The conductors of $L$ and $\phi(L)$ coincide (even though their discriminants do not). Moreover, we relate the discriminants of $K$ and $\phi(K)$ using the central inertia of $L$ (see Definition \ref{central inertia}). We prove (see Proposition \ref{prop:discrelation})
\begin{equation*}
\begin{array}{rcl}
|\Disc(K)|>\CC(L)^{1/2}&\Rightarrow&|\Disc(\phi(K))|<\CC(\phi(L))^{1/2},
\end{array}
\end{equation*}
and we use this phenomenon to obtain $\NumD^{(r_2)}(X)$ from the
asymptotic number of $D_4$-fields ordered by conductor whose quadratic
subfield has small discriminant by employing a simple sieve.

The proof of Theorem \ref{MAINTHEOREM} does ultimately rely on both
the analytic techniques of \cite{CDOQuartic} used to count
$D_4$-fields by discriminant as well as the geometry-of-numbers
methods used to count $S_4$-fields as in
\cite{BhargavaQuarticCount}. Either can be used to obtain asymptotics
for $D_4$-fields of bounded conductor whose quadratic subfield has
small discriminant, but the sieve used to determine the asymptotics of
$\NumD^{(r_2)}(X)$ from counting such $D_4$-fields requires two
ingredients: first, uniform error estimates on the number of such
$D_4$-fields having large central inertia, and second, asymptotics for
the number of such $D_4$-fields with prescribed ramification
conditions. We are able to obtain the former using analytic methods
and the latter using geometry-of-numbers techniques. In fact, this
method of proof allows us to count $D_4$-fields with prescribed
splitting and ramification conditions yielding Theorems
\ref{congruence conditions} and \ref{classgroups}. Additionally,
Theorem \ref{congruence conditions} in conjunction with Theorem
\ref{thm:analyticmain2} and the computations of $p$-adic densities (see
Proposition \ref{propdenmax}) implies Theorem \ref{stablefamilies}.

\smallskip 

We remark that we obtain three different
interpretations for the Euler product
\begin{equation*}
\prod_p\left(1-\frac{1}{p^2}+\frac{2}{p^3}+\frac{2}{p^4}\right).
\end{equation*}
It arises in the residue of a double Dirichlet series defined in Section 3, the contribution from diagonal terms in the sum over quadratic fields $K$ of
$\frac{L(1,K/\bQ)}{L(2,K/\bQ)}$ weighted by $\Disc(K)^{-1}$  in
Section~5, and the $p$-adic volumes of a coregular representation of a non-reductive group in Section 6.

\smallskip

This paper is organized as follows. First, we summarize
basic arithmetic properties of $D_4$-fields and their invariants in Section 2, including a table of the splitting behavior of primes in $D_4$-fields depending on their inertia and decomposition group. We explicitly relate the conductors and quadratic
discriminants of $D_4$-fields $L$ and $\phi(L)$. Next in Section 3, we
further develop the method of
\cite{bhargavamass,kedlayamass,melaniemass} and simultaneously obtain
heuristics for families of $D_4$-fields ordered in multiple different
ways. We begin counting $D_4$-fields with bounded invariants in
Section 4. Using the analytic methods of \cite{CDOQuartic}, we obtain
asymptotics for the number of such fields with small quadratic
discriminant in terms of a sum of ratios of $L$-values. By isolating the diagonal terms, we prove Theorem \ref{2} in
Section 5. In Section 6, we recall Wood's parametrization of quartic rings with a
quadratic subrings and adapt it to obtain a
modified parametrization of $D_4$-fields in terms of certain integral orbits of
a coregular representation $V$ for a non-reductive group $G$. We study the
$p$-adic properties of this representation, including those arising from the outer automorphism $\outeraut$. We then use
geometry-of-numbers methods in Section 7 to count integral orbits of
$G$ on $V$ having bounded invariants.  Using the results and methods
from Sections 4, 6, and 7, we obtain crucial uniformity estimates in
Section 8 that will be necessary to carry out the various requisite
sieves. Finally, in Section 9, we prove the main theorems by using the
analytic results of Sections 4, 7, and 8, in conjunction with the
algebraic properties of the outer automorphism $\phi$ proved in Sections 2 and 6.


\section{General properties of $D_4$-fields}\label{sec:d4fields}

Recall that $D_4$ denotes the order-8 group of symmetries of a square, and
a {\it $($quartic$)$ $D_4$-field} is a degree-$4$
field extension of $\Q$ whose normal closure has Galois group $D_4$ over
$\Q$. We let $\rotation$ denote a $90^\circ$-rotation of a square and $\reflection$ denote a reflection of a square so that
\begin{equation*} D_4 =
\Bigl\langle  \rotation ,\ \reflection  \ \mid \ \rotation^4 = 1,\  \reflection^2 = 1, \ \reflection^{-1}\rotation\reflection =
\rotation^3  \Bigr\rangle. \end{equation*}
The group $D_4$ has nontrivial center
$Z(D_4) = \{ 1, \rotation^2 \}$.

\subsection{Automorphisms of $D_4$ and the Galois theory of $D_4$-fields}\label{sec:subgroups}

We first describe important group-theoretic properties of $D_4$ as well as their applications to $D_4$-fields via Galois theory. Recall that the inner automorphism group $D_4/Z(D_4)$ is isomorphic to the Klein four group $V_4$, but the full automorphism group is isomorphic to $D_4$ (see pgs.~83--85 of \cite{rose}). The non-trivial outer automorphism $\outeraut$ of $D_4$ has order $4$ and can be described explicitly by
	\begin{equation}\outeraut(\rotation) = \rotation \qquad \mbox{and} \qquad \outeraut(\reflection) = \rotation\reflection.\end{equation}

\medskip

Let $L_1$ be a $D_4$-field, and denote its normal closure by $M$ so that $\Gal(M/\Q) = D_4$. Below, we describe a subfield diagram of $M$ corresponding to the subgroup lattice of $D_4$:
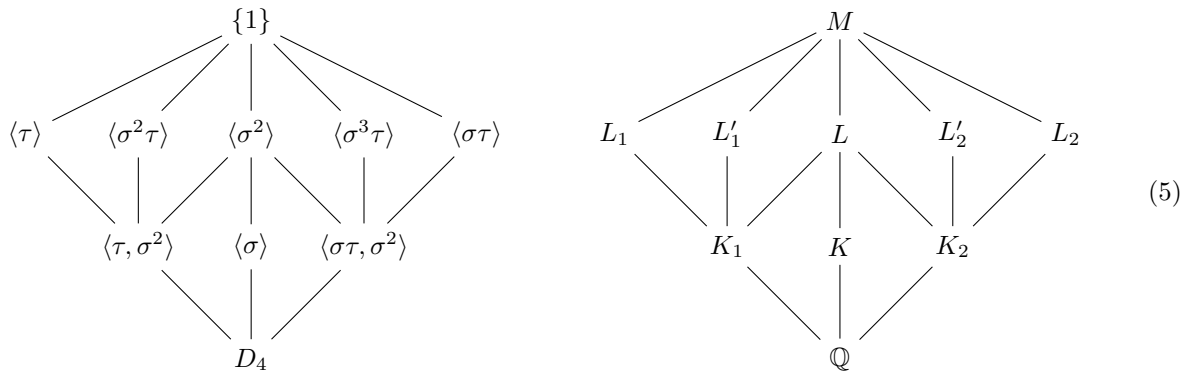
\begin{equation}\label{d4diagram}
\begin{array}{cccc}
\begin{tikzpicture}
[align=center,node distance=1.5cm]
\node (F) {$D_4$};
\node (K) [above of=F, left of=F] {$\langle \reflection, \rotation^2 \rangle$};
\node (Sigma) [above of=F] {$\langle \rotation \rangle$};
\node (Kprime) [above of=F, right of=F] {$\langle \rotation\reflection, \rotation^2 \rangle$};
\node (Z) [above of=Sigma] {$\langle \rotation^2 \rangle$};
\node (L) [above of=K, left of=K] {$\langle \reflection \rangle $};
\node (Lminus) [above of=K] {$\langle \rotation^2 \reflection \rangle $};
\node (Lprime) [above of=Kprime, right of=Kprime] {$ \langle \rotation\reflection \rangle $};
\node (Lminusprime) [above of=Kprime] {$ \langle \rotation^3 \reflection \rangle $};
\node (Ltilde) [above of=Z] {$\{ 1 \}$};

\draw [-] (F) to node {} (K);
\draw [-] (F) to node {} (Kprime);
\draw [-] (F) to node {} (Sigma);
\draw [-] (K) to node {} (L);
\draw [-] (K) to node {} (Lminus);
\draw [-] (Kprime) to node {} (Lprime);
\draw [-] (Kprime) to node {} (Lminusprime);
\draw [-] (Sigma) to node {} (Z);
\draw [-] (K) to node {} (Z);
\draw [-] (Kprime) to node {} (Z);
\draw [-] (L) to node {} (Ltilde);
\draw [-] (Lminus) to node {} (Ltilde);
\draw [-] (Lprime) to node {} (Ltilde);
\draw [-] (Lminusprime) to node {} (Ltilde);
\draw [-] (Z) to node {} (Ltilde);
\end{tikzpicture} &&& \begin{tikzpicture}
[align=center,node distance=1.5cm]
\node (F) {$\Q$};
\node (K) [above of=F, left of=F] {$K_1$};
\node (Sigma) [above of=F] {$K$};
\node (Kprime) [above of=F, right of=F] {$K_2$};
\node (Z) [above of=Sigma] {$L$};
\node (L) [above of=K, left of=K] {$L_1$};
\node (Lminus) [above of=K] {$L_1'$};
\node (Lprime) [above of=Kprime, right of=Kprime] {$L_2$};
\node (Lminusprime) [above of=Kprime] {$L_2'$};
\node (Ltilde) [above of=Z] {$M$};

\draw [-] (F) to node {} (K);
\draw [-] (F) to node {} (Kprime);
\draw [-] (F) to node {} (Sigma);
\draw [-] (K) to node {} (L);
\draw [-] (K) to node {} (Lminus);
\draw [-] (Kprime) to node {} (Lprime);
\draw [-] (Kprime) to node {} (Lminusprime);
\draw [-] (Sigma) to node {} (Z);
\draw [-] (K) to node {} (Z);
\draw [-] (Kprime) to node {} (Z);
\draw [-] (L) to node {} (Ltilde);
\draw [-] (Lminus) to node {} (Ltilde);
\draw [-] (Lprime) to node {} (Ltilde);
\draw [-] (Lminusprime) to node {} (Ltilde);
\draw [-] (Z) to node {} (Ltilde);
\end{tikzpicture}\\
\end{array}
\end{equation}

Here, a subgroup $G$ of $D_4$ and a subfield $F$ of $M$ in the same position are related by $\Gal(M/F) = G$. The
fields $L_i'$ are the (unique) Galois conjugates of $L_i$ for $i = 1$ or
$2$, and $L$ is the unique quartic Galois subfield of $M$.
While $L_1$ and $L_2$ are {\em not} conjugate, the
outer automorphism $\outeraut$
maps $\Gal(M/L_1) \to \Gal(M/L_2)$ and $\Gal(M/L_1') \to \Gal(M/L_2')$. However, it sends $\Gal(M/L_2) \to \Gal(M/L_1')$ and $\Gal(M/L_2') \to \Gal(M/L_1)$.
It also interchanges $\Gal(M/K_1)$ and $\Gal(M/K_2)$ while leaving $\Gal(M/K)$ fixed.

\begin{defn}\label{defphi}
  If $L_1$ is a $D_4$-field with Galois closure $M$ over $\bQ$, then
  we denote by $\phi(L_1)$ the quartic subfield of $M$ fixed by
  $\phi(\Gal(M/L_1))$. If $K_1$ denotes the quadratic subfield of
  $L_1$, then we denote by $\phi(K_1)$ the quadratic subfield of $M$
  fixed by $\phi(\Gal(M/K_1))$.
\end{defn}

\noindent In the notation of~\eqref{d4diagram}, we have $\phi(L_1) = L_2$ and $\phi(K_1)
 = K_2$. 

\subsection{Arithmetic of $D_4$-fields}\label{sec:initiald4fields}
We now describe the splitting behavior of primes in octic
$D_4$-fields $M$ and their subfields. 
\begin{defn}
If $F$ is a number
field, then the {\em splitting type} $\varsigma_p(F)$ at $p$ of $F$ satisfies
	$$\varsigma_p(F) = (f_1^{e_1}f_2^{e_2} \hdots) \quad \Leftrightarrow \quad \cO_F/p\cO_F \cong \bF_{p^{f_1}}[t_1]/(t_1^{e_1}) \oplus \bF_{p^{f_2}}[t_2]/(t^{e_2}) \oplus \hdots$$
Similarly, if $R$ is a ring, then the {\em splitting type} $\varsigma_p(R)$ at $p$ is equal to $(f_1^{e_1}f_2^{e_2} \hdots)$ if and only if $$R/pR \cong \bF_{p^{f_1}}[t_1]/(t_1^{e_1}) \oplus \bF_{p^{f_2}}[t_2]/(t^{e_2}) \oplus \hdots$$ 
\end{defn}
Let $D_p$ denote the decomposition group of $p$ in $\Gal(M/\bQ)$, and let $I_p$ denote the inertia subgroup of $D_p$. For an arbitrary prime $p$, we determine the splitting type of $M$ and all of its subfields using the notation described in \eqref{d4diagram} depending on the choices for $D_p$ and $I_p$ in the table below.

\medskip

In Table~1, the first group consists of unramified splitting types, the second and third
groups consist of tamely ramified splitting types, and the fourth group
consists of wildly ramified splitting types. In particular, the splitting type
of an odd prime $p$ must appear in the first three groups of
Table~1. We distinguish between the tamely ramified splitting types depending on whether the center $\langle \sigma^2 \rangle$ of $D_4$ is contained in $I_p$.

\begin{center}
\begin{tabular}{|c | c||| c|| c| c|| c| c|| c| c|||c|c|}
  \hline &&&&&&&&&\multicolumn{2}{|c|}{}\\
  $I_p$ & $D_p$ & $\varsigma_p(M)$ &$\varsigma_p(L_1)$ &$\varsigma_p(K_1)$ &
  $\varsigma_p(L_2)$ &$\varsigma_p(K_2)$ &
  $\varsigma_p(L)$ &$\varsigma_p(K)$ & \multicolumn{2}{|c|}{}
	  	\\[12pt] \hline &&&&&&&&&&\\[-10.5pt]
  		\hline &&&&&&&&&&\\[-10.5pt]
  		\hline &&&&&&&&&&\\[-5pt]
  $\{1\}$&$\{1\}$&$(11111111)$&$(1111)$&$(11)$&$(1111)$&$(11)$&$(1111)$&$(11)$&\multirow{5}{*}{\rotatebox[origin=c]{270}{\small Unramified}}&\multirow{10}{*}{\rotatebox[origin=c]{270}{\small Lacks central inertia}}\\
  $\{1\}$&$\langle\rotation^2\rangle$&$(2222)$&$(22)$&$(11)$&$(22)$&$(11)$&$(1111)$&$(11)$&&\\
  $\{1\}$&$\langle\rotation\reflection\rangle$&$(2222)$&$(22)$&$(2)$&$(112)$&$(11)$&$(22)$&$(2)$&&\\
  $\{1\}$&$\langle\reflection\rangle$&$(2222)$&$(112)$&$(11)$&$(22)$&$(2)$&$(22)$&$(2)$&&\\
  $\{1\}$&$\langle\rotation\rangle$&$(44)$&$(4)$&$(2)$&$(4)$&$(2)$&$(22)$&$(11)$&&\\
		[5pt] \cline{1-10} &&&&&&&&&&\\[-4pt]

  $\langle\reflection\rangle$&$\langle\reflection\rangle$&
  $(1^21^21^21^2)$&$(1^211)$&$(11)$&$(1^21^2)$&$(1^2)$&$(1^21^2)$&$(1^2)$&\multirow{4}{*}{\rotatebox[origin=c]{270}{\small Tame}}&\\

  $\langle\reflection\rangle$&$\langle\reflection,\rotation^2\rangle$&
  $(2^22^2)$&$(1^22)$&$(11)$&$(2^2)$&$(1^2)$&$(1^21^2)$&$(1^2)$&&\\

  $\langle\rotation\reflection\rangle$&$\langle\rotation\reflection\rangle$&
  $(1^21^21^21^2)$&$(1^21^2)$&$(1^2)$&$(1^211)$&$(11)$&$(1^21^2)$&$(1^2)$&&\\

  $\langle\rotation\reflection\rangle$&$\langle\rotation\reflection,\rotation^2\rangle$&
  $(2^22^2)$&$(2^2)$&$(1^2)$&$(1^22)$&$(11)$&$(1^21^2)$&$(1^2)$&&\\
 	[5pt] \hline &&&&&&&&&&\\[-10.5pt]\hline &&&&&&&&&&\\[-4pt]

  $\langle\rotation\rangle$&$\langle\rotation\rangle$&
  $(1^41^4)$&$(1^4)$&$(1^2)$&$(1^4)$&$(1^2)$&$(1^21^2)$&$(11)$&\multirow{6}{*}{\rotatebox[origin=c]{270}{\small Tame}}&\multirow{10}{*}{\rotatebox[origin=c]{270}{\small Has central inertia}}\\

  $\langle\rotation\rangle$&$D_4$&
  $(2^4)$&$(1^4)$&$(1^2)$&$(1^4)$&$(1^2)$&$(2^2)$&$(1^2)$&&\\

  $\langle\rotation^2\rangle$&$\langle\rotation^2\rangle$&
  $(1^21^21^21^2)$&$(1^21^2)$&$(11)$&$(1^21^2)$&$(11)$&$(1111)$&$(11)$&&\\

  $\langle\rotation^2\rangle$&$\langle\reflection,\rotation^2\rangle$&
  $(2^22^2)$&$(1^21^2)$&$(11)$&$(2^2)$&$(2)$&$(22)$&$(2)$&&\\

  $\langle\rotation^2\rangle$&$\langle\rotation\reflection,\rotation^2\rangle$&
  $(2^22^2)$&$(2^2)$&$(2)$&$(1^21^2)$&$(11)$&$(22)$&$(2)$&&\\

  $\langle\rotation^2\rangle$&$\langle\rotation\rangle$&
  $(2^22^2)$&$(2^2)$&$(2)$&$(2^2)$&$(2)$&$(22)$&$(11)$&&\\
 	[5pt] \cline{1-10} &&&&&&&&&&\\[-4pt]

  $\langle \reflection, \rotation^2 \rangle$ & $\langle \reflection, \rotation^2 \rangle$ &
  $(1^41^4)$ & $(1^21^2)$ & $(11)$ & $(1^4)$ & $(1^2)$ & $(1^21^2)$ & $(1^2)$& \multirow{3}{*}{\rotatebox[origin=c]{270}{\small Wild}}&\\

  $\langle \rotation\reflection, \rotation^2 \rangle$ & $\langle \reflection, \rotation^2 \rangle$ &
  $(1^41^4)$ & $(1^4)$ & $(1^2)$ & $(1^21^2)$ & $(11)$ & $(1^21^2)$ & $(1^2)$ &&\\

  $D_4$ & $D_4$ &
  $(1^8)$ & $(1^4)$ & $(1^2)$ & $(1^4)$ & $(1^2)$ & $(1^4)$ & $(1^2)$&&\\[-4pt] &&&&&&&&&&\\
\hline
\end{tabular}\\[.1in]
Table 1: Splitting type for a given decomposition and inertia group.
\end{center}
\begin{defn}\label{central inertia}
  We say that a $D_4$-field $L_1$ (or a pair $(L_1,K_1)$ consisting of a $D_4$-field $L_1$ and its quadratic subfield $K_1$)
  {\em has central inertia at $p$} when $I_p$ contains $\sigma^2$ or equivalently, when
   the pair $(\varsigma_p(L_1),\varsigma_p(K_1))$ is 
  $((1^21^2),(11))$, $((2^2),(2))$, or $((1^4),(1^2))$.
\end{defn}
Note that if $M$ and $K$ are as in \eqref{d4diagram}, the extension $M/K$ is ramified at a prime $p$ if and only if $(L_1,K_1)$ has central inertia at $p$. 

By analogy, we define splitting types at $\infty$. If $K$ is any quadratic field and $L$ is any quartic field, we write
	\begin{equation}\begin{array}{ccc}
	\varsigma_\infty(K) := \begin{cases} (11) & \mbox{ if $K \otimes \bR = \bR^2$,} \\
											(2) & \mbox{ if $K \otimes \bR = \bC$;} 
											\end{cases}
											& \qquad &
	\varsigma_\infty(L) := \begin{cases} (1111) & \mbox{ if $L \otimes \bR = \bR^4$,} \\
											(112) & \mbox{ if $L \otimes \bR = \bR^2 \oplus \bC$,} \\
											(22) & \mbox{ if $L \otimes \bR = \bC^2$.} 
											\end{cases} 
  \end{array}\end{equation}

\subsection{Invariants of $D_4$-fields}
Next, we compare the Artin conductor of a $D_4$-field
$L_1$ to the discriminant of $L_1$ as well as the
products of the discriminants of certain subfields of the normal
closure of $L_1$. Additionally, we define two fundamental invariants and a slightly refined conductor that partially recovers the splitting type of $L_1$ at $\infty$.

\medskip

If $\Gal(\overline{\Q}/\Q)$ denotes the absolute Galois group of $\Q$, and $M$ is
the normal closure of $L_1$ as in \eqref{d4diagram}, then there is a (unique up to conjugacy)
irreducible 2-dimensional Galois representation
\begin{equation*}
\rho_M: \Gal(\overline{\Q}/\Q)\to\GL_2(\C)
\end{equation*}
that factors through $\Gal(M/\bQ)$. It arises as the composition of
$\Gal(\overline{\Q}/\Q) \twoheadrightarrow \Gal(M/\bQ)$ and the unique 2-dimensional
irreducible representation of $D_4\cong \Gal(M/\bQ)$. We let $\Cond(\rho_M)$ denote the
Artin conductor of $\rho_M$. This invariant can be described in
terms of the discriminant of the quadratic subfield $K_1$ and
$\Nm_{K_1}(\Disc(L_1/K_1))$, the image under the norm map of $K_1$ of the relative discriminant of $L_1$ over $K_1$:

\begin{prop}\label{cond}
Let $L_1$ denote a $D_4$-field with normal closure $M$, and let $K_1$ be its quadratic subfield. We then have:
\begin{equation*}
\Cond(\rho_M)=|\Disc(K_1)\cdot\Nm_{K_1}(\Disc(L_1/K_1))|.
\end{equation*}
\end{prop}
\begin{proof}
The proposition will follow from the fact that the representation $\Ind_{\langle \reflection \rangle}^{D_4} \bOne$ of $D_4$ decomposes into a direct sum of $\Ind_{\langle \reflection, \rotation^2\rangle}^{D_4} \bOne$ and the irreducible 2-dimensional representation of $D_4$. To prove this fact, first note that each coset of $\langle \reflection \rangle$ in $D_4$ contains a unique power of $\rotation$, so we can represent $\Ind_{\langle \reflection \rangle}^{D_4}$ in terms of the basis $\left\langle[1], [\rotation], [\rotation^2], [\rotation^3]\right\rangle$. We can then decompose $\Ind_{\langle \reflection \rangle}^{D_4} 
\bOne = V_1 \oplus V_2$ where
\[ V_1 = \left\langle [1] + [\rotation^2], [\rotation] + [\rotation^3] \right\rangle; \qquad
   V_2 = \left\langle [1] - [\rotation^2], [\rotation] - [\rotation^3] \right\rangle. \]
Since $\rotation$ swaps the two basis elements of $V_1$ while $\reflection$ and $\rotation^2$ act trivially, $V_1$ can be identified with $\Ind_{\langle\reflection,\rotation^2\rangle}^{D_4} \bOne$. Furthermore, one can see that $V_2$ is irreducible, and it is well-known that there is a unique irreducible 2-dimensional representation of $D_4$.

Now, if $M$ is the normal closure of $L_1$ as in \eqref{d4diagram} and
$\rho$ denotes the Galois representation constructed by composing
$\Ind_{\langle \reflection \rangle}^{D_4} \bOne$ with $\Gal(\overline{\bQ}/\bQ)
\twoheadrightarrow \Gal(M/\bQ)$, then its Artin conductor satisfies
$$\Cond(\rho) = |\Disc(L_1)|.$$ We can compute $\Cond(\rho)$ as a
product of the conductors of its subrepresentations: the Galois
representation $\Ind_{\langle \reflection, \rotation^2\rangle}^{D_4}
\bOne\circ\left(\Gal(\overline{\bQ}/\bQ) \twoheadrightarrow \Gal(M/\bQ) \right)$ has
conductor $\Disc(K_1)$, so we obtain
	$$|\Disc(L_1)| = |\Disc(K_1)|\cdot\Cond(\rho_M).$$
The relative discriminant formula implies that
$|\Disc(L_1)|=\Disc(K_1)^2\Nm_{K_1}(\Disc(L_1/K_1)),
$ and so we conclude the proposition.
\end{proof}

\begin{defn}\label{condD4} The {\em $($signed$)$ conductor} $\CC(L_1)$ of a $D_4$-field $L_1$ whose quadratic subfield is denoted by $K_1$ is defined as
	$$\CC(L_1) \ := \ \frac{\Disc(L_1)}{\Disc(K_1)}.$$
\end{defn}

From the definition of the conductor and Proposition \ref{cond}, it follows immediately that two
$D_4$-fields $L_1$ and $L_2$ with the same normal closure $M$ have the same
conductor. Furthermore, if $L_1$ has no central inertia, then $\CC(L_1) = \Disc(K_1) \cdot
\Disc(\phi(K_1))$. More precisely, 
if $L$ is a number field and $p$ is a prime number, let
$\Disc_p(L)$ denote the $p$-part of the discriminant, and  let $\CC_p(L)$ be the $p$-part of the conductor. We then have:
\begin{prop}\label{prop:discrelation}
If $L_1$ is a $D_4$-field with quadratic subfield $K_1$, then for all odd primes $p$:
\begin{equation*}
\CC_p(L_1) = \begin{cases} p^2 \cdot \Disc_p(K_1)\cdot\Disc_p(\phi(K_1)) &\mbox{if $I_p = \langle \sigma^2 \rangle$;} \\
 \Disc_p(K_1)\cdot\Disc_p(\phi(K_1)) &\mbox{ otherwise.}
\end{cases}\end{equation*}
\end{prop}
\begin{proof}
We refer to the notation described in \eqref{d4diagram}, where $\phi(K_1) = K_2$. Table~1 shows that if $I_p \neq \langle
\rotation^2\rangle$, then $\Disc_p(K_2) =
\Nm_{K_1}(\Disc_p(L_1/K_1))$.  Thus, $\CC_p(L_1)=
\Disc_p(K_1)\cdot\Disc_p(K_2).$ However, when $I_p = \langle
\rotation^2\rangle$, Table 1 implies that $\Disc_p(K_1) =
\Disc_p(K_2) = 1$, but $\Nm_{K_1}(\Disc_p(L_1/K_1)) =
\Nm_{K_2}(\Disc_p(L_2/K_2)) = p^2$. Thus, we have that $\CC_p(L_1) =
p^2 \cdot\Disc_p(K_1)\cdot\Disc_p(K_2).$
\end{proof}

\noindent We are now ready to define the two fundamental invariants of a
$D_4$-field.
\begin{defn}\label{fundinv} If $L_1$ is a $D_4$-field with quadratic subfield $K_1$, define the {\em fundamental invariants} of $L_1$ as
\begin{equation*}
\q(L_1)\defeq\displaystyle\frac{\Disc(L_1)}{\Disc(K_1)^2} \qquad \mbox{ and } \qquad \D(L_1)\defeq \Disc(K_1).
\end{equation*}\end{defn}

\begin{rmk}\label{fundamental}For a $D_4$-field $L_1$, there is a global restriction on the integers $\q(L_1)$ and $\D(L_1)$, namely that they are each congruent to $0$ or $1$ mod $4$.\end{rmk}

Proposition \ref{cond} can be reformulated as $\Cond(\rho_M) =
|\q(L_1)\cdot\D(L_1)| = |\q(L_2)\cdot\D(L_2)|$ for a $D_4$-field $L_1$ as in \eqref{d4diagram}. Define
\begin{equation}\label{ci}
\ci(L_1) \defeq \frac{\CC(L_1)}{\Disc(K_1)\cdot\Disc(\phi(K_1))} = \left|\frac{\q(L_1)}{\D(\phi(L_1))}\right|,
\end{equation}
and for a prime $p$, let $\ci_p(L_1)$ denote the $p$-part of $\ci(L_1)$.
Proposition \ref{prop:discrelation} determines that for an odd prime $p$, $\ci_p(L_1)$ is equal to $p^2$ if and only if the inertia group $I_p$ at $p$ is equal to $\langle \sigma^2\rangle \subset D_4$. Note that it is always true that $\ci(L_1)=\ci(\phi(L_1))$.
Furthermore, it will not be necessary to compute $J_2$; it will be enough that
$J_2$ is absolutely bounded.

\section{Heuristics for counting $D_4$-fields by conductor}\label{heuristics}

In \cite{bhargavamass}, Bhargava developed heuristics for the
asymptotics of the number of $S_n$-fields of degree $n$ ordered by
discriminant. The framework used to formulate these heuristics was
expanded by Kedlaya \cite{kedlayamass} for families of Galois representations ordered by their Artin conductor. Additionally, Wood
\cite{melaniemass} predicted asymptotics (including mass formulae for the constants) for fixed-degree families of number fields whose normal closures have a fixed Galois group when such fields are ordered by invariants including the conductor. In this section, we adapt their heuristics to 
the family of $D_4$-fields ordered by our two fundamental invariants, $\q$ and $\D$ (see Definition \ref{fundinv}). We recover the predictions in \cite{melaniemass} for the number of $D_4$-fields ordered either by conductor or discriminant, and we additionally verify that the conjectured mass formula when ordering by discriminant is {\em not} equal to the constant $c$ determined by Cohen-Diaz y Diaz-Olivier in \cite{CDOQuartic}.

\subsection{The expected number of $D_4$-fields for a given pair of fundamental invariants}
Let $v$ be a place of $\Q$, and let $K_v\subset L_v$ be \'etale
algebras of $\Q_v$ of degrees $2$ and $4$, respectively. When $v$
corresponds to a finite prime $p$ (resp.\ infinity), we say that such a
pair $(L_v,K_v)$ is {\it compatible} with a pair
of integers $(\q,\D)$ if the $p$-parts (resp.\ signs) of $\Nm_{K_v}(\Disc(L_v/K_v))$, the norm in
$\Q_v$ of the relative discriminant of $L_v$ over $K_v$, and $\Disc(K_v)$, the
discriminant of $K_v$, agree with the $p$-parts (resp.\ signs) of $\q$
and $\D$, respectively. Note that when $\q>0$,
$\D$ can be positive or negative; however, when $\q<0$, $\D$ must be
positive, otherwise no such compatible pairs $(L_\infty,K_\infty)$ exist.

Given a place $v$ of $\Q$ and integers $\q$ and $\D$, let
$\Sigma_v(\q,\D)$ denote the set of pairs of $\Q_v$-algebras
$(L_v,K_v)$ that are compatible with $(\q,\D)$.
Let the {\em weighted local mass} $E_v(\q,\D)$ be defined by
  \begin{equation*}
    E_v(\q,\D)\defeq\sum_{(L_v,K_v)\in
  \Sigma_v(\q,\D)}\frac{1}{\#\Aut(L_v,K_v)},
  \end{equation*}
where $\Aut(L_v,K_v)$ is
the group of
automorphisms of $L_v$ that restrict to endomorphisms of the subalgebra $K_v$.  The following result evaluates $E_v(\q,\D)$ for all places $v$.
\begin{prop}\label{prop:expectednum} 
We have:
\begin{itemize}
\item[{\rm (1)}] If $\q$ and $\D$ are nonzero, then $E_\infty(\q,\D)=1/4$ when at least one of $\q$ or $\D$ is positive.
\item[{\rm (2)}] If $v$ corresponds to an odd
  prime $p$, then $E_p(\q,\D)$ is nonempty if and only if the
  $p$-parts of $(\q,\D)$ are one of $(1,1)$, $(p,1)$, $(p^2,1)$, $(1,p)$, or
  $(p,p)$. In each case, we have $E_p(\q,\D)=1$.
\item[{\rm (3)}] The values of
  $E_2(\q, \D)$ are given below.
\begin{equation*}
  \begin{tabular}{|c||c|c|c|}\hline
$\q$ & $\D = 1$ & $\D = 2^2$ & $\D = 2^3$ \\\hline
$1$  & $1$      & $1$      & $2$      \\
$2^2$  & $1$      & $1$      & $2$      \\
$2^3$  & $2$      & ---      & ---      \\
$2^4$ & $2$      & $2$      & $4$      \\
$2^5$ & $2$      & $4$      & $8$      \\
$2^6$ & $4$      & ---      & ---      \\\hline
  \end{tabular}
\end{equation*}
\begin{center}
\vspace{-7pt}\mbox{\rm Table 2: The value of $E_2(\q, \D)$}
\end{center}
\end{itemize}
\end{prop}
\begin{proof}
  The proposition follows from a direct computation using a database
  of local fields away from $p = 2$. The code for the case $p = 2$ using \cite{LMFDB, JJDRFields} can be found at
    \url{http://github.com/khwilson/D4Counting}.
\end{proof}

The framework in \cite{bhargavamass,melaniemass} depends on the basic heuristic assumption that for a family of number fields of fixed degree and fixed associated Galois group, the completions at different places behave independently of one another. This implies that the expected number of such number fields having given
invariants is equal to the infinite product over all places $v$ of $\Q$ of the weighted number of local extensions of $\bQ_v$ that are compatible with those invariants. More precisely:
\begin{conj} \label{ass}If $E(\q,\D)$ denotes the expected
number of isomorphism classes of $D_4$-fields with fundamental invariants equal to $\q$ and $\D$, we assume \begin{equation}\label{asseq}
  E(\q,\D)=\frac{1}{2}\cdot E_{\infty}(\q,\D)\cdot\prod_p E_p(\q,\D). 
\end{equation}
\end{conj}
The extra factor of $\frac{1}{2}$ above arises from two issues: (1) there is a global restriction on the invariants $\q$ and $\D$ (see Remark \ref{fundamental}), which occurs $\frac{1}{4}$ of the time, and is not taken into account by the local masses, and (2) the product of the local masses $E_v(\q,\D)$ determines the expected {\em weighted} number of $D_4$-fields with invariants $\q$ and $\D$, where a $D_4$-field $L$ is weighted by $\#\Aut(L)^{-1} = \frac{1}{2}$.

\subsection{Predicting the global distribution of $D_4$-fields using double Dirichlet series}
To determine the asymptotics of $\sum E(\q, \D)$, we study the
behavior of the double Dirichlet series
\begin{equation*}
  \xi(s,t) := \sum_{\D} \sum_{\q} \frac{E(\q,
    \D)}{|\q|^s|\D|^t},
\end{equation*}
which converges absolutely for $s, t > 1$. Since there are three possible
sign configurations for the pair of integers $(\q,\D)$, the archimedean contribution to $\xi(s,t)$ is exactly
$3/4$.  Additionally, $2\cdot E(\q,\D)$ is multiplicative with respect to both
$\q$ and $\D$, and so it follows from Proposition \ref{prop:expectednum}
that $\xi(s,t)$ can be expressed as
\begin{equation*}
\xi(s,t)=\frac{3}{8}\cdot\prod_{p}\xi_p(s,t),
\end{equation*}
where
\begin{equation*}
  \xi_p(s,t) =\Bigl(1+\frac{1}{p^s}+\frac{1}{p^{2s}}+
  \frac{1}{p^{t}}\Bigl(1+\frac{1}{p^{s}}\Bigr)\Bigr)
\end{equation*}
when $p$ is odd, and
\begin{equation*}
  \xi_2(s,t)=\Bigl(1+\frac{1}{2^{2s}}+\frac{2}{2^{3s}}
  +\frac{2}{2^{4s}}+\frac{2}{2^{5s}}+\frac{4}{2^{6s}}+
  \frac{1}{2^{2t}}\Bigl(1+\frac{1}{2^{2s}}+\frac{2}{2^{4s}}
  +\frac{4}{2^{5s}}\Bigr)+
  \frac{2}{2^{3t}}\Bigl(1+\frac{1}{2^{2s}}+\frac{2}{2^{4s}}
  +\frac{4}{2^{5s}}\Bigr)\Bigr).
\end{equation*}
 Define the correction factor at $2$ to be
\begin{equation*}
\widetilde{\xi}_2^{}(s,t)\defeq\xi_2(s,t)/\Bigl(1+\frac{1}{2^s}+\frac{1}{2^{2s}}+
  \frac{1}{2^{t}}\Bigl(1+\frac{1}{2^{s}}\Bigr)\Bigr).
\end{equation*}
We can rewrite $\xi(s, t)$ as
\begin{equation*}
  \begin{array}{rcl}
\displaystyle\xi(s,t)&=&\displaystyle\frac{3}{8}\cdot\widetilde{\xi}_2(s,t)\cdot\prod_p\Bigl(1+\frac{1}{p^s}+\frac{1}{p^{2s}}+
\frac{1}{p^{t}}\Bigl(1+\frac{1}{p^{s}}\Bigr)\Bigr)\\[.2in]
&=&
\displaystyle\frac{3}{8}\cdot\widetilde{\xi}_2(s,t)\cdot\zeta(s)\cdot\zeta(t)\cdot\prod_p
(1-p^{-2t}-p^{-t-2s}-p^{-3s}+p^{-2t-2s}+p^{-t-3s}).
  \end{array}
\end{equation*}
Therefore, the function
$\xi(s,t)$ is holomorphic in the region
$t>\frac{1}{2},\;s>\frac{1}{3}$ aside from poles at the lines $s=1$ and $t=1$.

\subsection{Heuristics} We now consider families of $D_4$-fields under
different orderings. If $X$ and $Y$ be positive real numbers
going to infinity, let $E_{\q,\D}(X,Y)$ denote the expected number of isomorphism classes of
$D_4$-fields $L$ such that $|\q(L)|<X$ and $|\D(L)|<Y$, i.e.
	$$E_{\q,\D}(X,Y) \  := \ \sum_{{|\q| < X}\atop{|\D| < Y}} E(\q,\D).$$ 
Then, by
computing the residue of $\xi(s,t)$ at $(1,1)$, we obtain the
heuristic
\begin{equation}\label{eqHXY}
  E_{\q,\D}(X,Y) \sim
  \frac{3}{8}\cdot\prod_p\Bigl(1-\frac1{p^2}-\frac2{p^3}+\frac2{p^4}\Bigr)\cdot X\cdot Y.
\end{equation}
Here, the correction factor $\widetilde{\xi}_2(1,1)= 1$. Note that this heuristic only relies on the Assumption \ref{ass}. If we were to take \eqref{asseq} as a definition, then we have completely verified the main term \eqref{eqHXY}, and we can additionally obtain a power-saving. 

\subsubsection*{Heuristics for the family of $D_4$-fields ordered
  by conductor}
Next, we consider the family of $D_4$-fields ordered by conductor. Let
$E_{\CC}(X)$ denote the expected number of isomorphism classes of $D_4$-fields $L$ such
that $|\CC(L)|<X$. If we let $E(\CC)$ denote the expected number of
$D_4$-fields with conductor $\CC$, then we have
\begin{equation*}
\sum_{\CC}\frac{E(\CC)}{|\CC|^s}=\xi(s,s)
\end{equation*}
since $\CC(L)=\q(L)\D(L)$. The function $\xi(s,s)$ has a double pole
at $1$ and, by computing its residue, we obtain 
\begin{equation}\label{eqCXY}
  E_{\CC}(X)\sim
  \frac{3}{8}\cdot\prod_p\Bigl(1-\frac1{p^2}-\frac2{p^3}+\frac2{p^4}\Bigr)\cdot X\log (X).
\end{equation}
The correction factor at $2$ is again $\widetilde{\xi}_2(1, 1) = 1$.

\subsubsection*{Heuristics for the family of $D_4$-fields ordered
  by discriminant}
Finally, we consider the family of $D_4$-fields ordered by
discriminant. Let $E_{\Disc}(X)$ denote the expected number of isomorphism classes of
$D_4$-fields $L$ such that $|\Disc(L)|<X$. If we let $E(\Disc)$ denote the
expected number of $D_4$-fields $L$ with discriminant equal to $\Disc$, then we have
\begin{equation*}
\sum_{\Disc}\frac{E(\Disc)}{|\Disc|^s}=\xi(s,2s)
\end{equation*}
since $\Disc(L)=\q(L)\D(L)^2$. The function $\xi(s,2s)$ has a simple pole
at $1$ and, by computing the residue, we obtain
\begin{equation}\label{eqDXY}
  E_{\Disc}(X)\sim \frac{3\cdot 11^2}{2^6\cdot
    17}\cdot\prod_p\Bigl(1+\frac1{p^2}-\frac1{p^3}-\frac1{p^4}\Bigr)\cdot X.
\end{equation}
In this case, the correction factor at $2$ is $\widetilde{\xi}_2(1,2) =
11^2/(17\cdot 2^3)$.

Cohen-Diaz y Diaz-Olivier showed in Proposition 6.2 of \cite{CDOQuartic}
that the number of $D_4$-fields having
discriminant bounded by $X$ is $\sim cX$ where $c \approx 0.052$,
whereas the constant on the right hand side of \eqref{eqDXY} is $\approx 0.406$. This implies that when ordering $D_4$-fields by discriminant, the completions of such fields at different primes do {\em not} behave independently of one another in the sense of \cite{bhargavamass}, and so Assumption \ref{ass} is false.

%

\section{Counting $D_4$-fields using analytic methods}

In this section, we obtain asymptotics for the number of $D_4$-fields,
ordered by conductor, whose quadratic subfield has small
discriminant, following the methods of Cohen-Diaz y Diaz-Olivier \cite{CDOQuartic} where similar
asymptotics for the number of such $D_4$-fields ordered by discriminant are determined. By refining their arguments, we are able to count $D_4$-fields of bounded conductor whose quadratic subfields have small discriminant and satisfy a prescribed set of splitting conditions at a finite number of primes. We begin with a few definitions before giving the precise statement of the main theorem of the section.
\begin{defn}\label{alicond}
  If $K$ is a quadratic field and $L$ is a quadratic extension of $K$,
  define the {\em conductor} of the pair $(L,K)$ as
\begin{equation}
\CC(L,K) \defeq \frac{\Disc(L)}{\Disc(K)}.
\end{equation}
\end{defn}
\noindent If $L$ is a $D_4$-field and $K$ denotes its (unique) quadratic subfield,
then $\CC(L,K) = \CC(L)$. 

\medskip
We refine the notion of a collection of local specifications described in the introduction. Let $\Sigmaall_v = \{ (\varsigma_v(L), \varsigma_v(K)) \}$ be the set of all pairs consisting of a possible splitting type for a place $v$ in a $D_4$-field $L$ and a consistent splitting type at $v$ for its quadratic subfield $K$. We refer to a collection $\Sigma
= (\Sigma_v)_v$ as a set of {\it local specifications} if for each $v$, $\Sigma_v \subseteq \Sigmaall_v$. 

\begin{defn}\label{stablefam}
A set of local specifications
$\Sigma= (\Sigma_v)_v$ is {\it stable} if for every prime $p$ and every quadratic
splitting type $\varsigma_p'$ (equal to either $(11)$, $(2)$, or $(1^2)$), the set $\Sigma_p$ either contains all possible
pairs $(\ast, \varsigma_p')$ or none of them.
\end{defn}
Additionally, we denote by $\LL(\Sigma)$ the set of $D_4$-fields $L$ with quadratic subfield $K$
such that $(\varsigma_v(L), \varsigma_v(K)) \in \Sigma_v$ for all $v$. Similarly, let $\KK(\Sigma)$ denote the set of quadratic subfields of $\LL(\Sigma)$. Note that when $\Sigma$ is stable, the set $\LL(\Sigma)$ consists of
all $D_4$-fields that are quadratic extensions of all the fields in
$\KK(\Sigma)$.

\medskip

For a set of local specifications $\Sigma$, let $\NumC(\Sigma;X, Y)$
be the number of isomorphism classes of $D_4$-fields $L \in \LL(\Sigma)$ such that $|\CC(L)| < X$ and $|\D(L)| < Y$. Additionally, set $\NumC(X,Y) :=\NumC(\Sigmaall;X,Y)$. In this
section, we compute asymptotics for $\NumC(\Sigma;X,X^\beta)$ when
$\Sigma$ is stable and $\beta<2/3$. More precisely, our goal is to prove
the following theorem:
\begin{thm}\label{thm:analyticmain2}
Let $\Sigma$ be a stable set of local specifications. Then, for every
$\beta<2/3$, we have
\begin{equation*}
  \NumC(\Sigma;X, X^{\beta})=\frac{1}{2\zeta(2)}
 \cdot \Biggl(\sum_{\substack{K\in\KK(\Sigma)\\ |\Disc(K)|<
      X^{\beta}}}\frac{L(1,K/\mathbb{Q})}{
    L(2,K/\mathbb{Q})}\cdot\frac{2^{-\complex(K)}}{|\Disc(K)|}\Biggr)\cdot X+o_{\beta}(X),
\end{equation*}
\end{thm}

We do so by first demonstrating that the number of quadratic extensions of quadratic number fields that are not $D_4$-quartic fields is negligible, thus we can compute $\NumC(\Sigma;X,X^\beta)$ in terms of these towers of quadratic extensions. In \cite{CDOQuartic}, the authors define a Dirichlet series for each quadratic field $K$ whose residue at $s = 1$ is shown to be equal to the number of quadratic extensions of $K$. We then carry out a smooth count for the quartic fields in $\LL(\Sigma)$ that are quadratic extensions of $K$ and subsequently obtain the theorem by summing over all $K \in \KK(\Sigma)$.

\subsection{Quadratic extensions of quadratic number fields}

If $L$ is a quadratic extension of a quadratic field $K$, then $L$ is
either a $D_4$-field or it is Galois with $\Gal(L/\bQ) = C_4$ or $V_4$. In the following
lemma, we prove a bound for the number of pairs $(L,K)$ having bounded
conductor, where $L$ is a Galois quartic field and $K$ is a quadratic
subfield of $L$ having small discriminant.

\begin{lemma}\label{thm:abelianbounds}
Let $\beta<1$ be fixed.  The number of pairs of $(L,K)$, where $L$ is
a Galois quartic field, $K$ is a quadratic subfield of $L$, the
conductor $\CC(L,K)<X$, and $|\Disc(K)|<X^\beta$ is bounded by
$O_\epsilon(X^{(1+\beta)/2+\epsilon})$.
\end{lemma}
\begin{proof}
Let $(L,K)$ be a pair satisfying the conditions of the lemma. By the
relative discriminant formula, we have
\begin{equation*}
|\Disc(L)|=|\Disc(K)\cdot\CC(L,K)|<X^{1+\beta}.
\end{equation*}
It is known from \S2.4 and \S2.5 of \cite{CDOSurvey} that the number of Galois
quartic fields whose discriminant have absolute value less than $X$ is
bounded by $O_\epsilon(X^{1/2+\epsilon})$. The lemma follows
immediately.
\end{proof}

For stable $\Sigma$, we can thus prove Theorem \ref{thm:analyticmain2} by counting the number
of quadratic extensions over quadratic fields in $\KK(\Sigma)$ whose relative
discriminants have bounded norm. To this end, we consider the Dirichlet series $\Phi_{K,2}(s) =
\Phi_{K,2}(C_2,s)$ defined in \cite{CDOQuartic} for any number field $K$ as
\begin{equation*}
\Phi_{K,2}(s)\defeq\sum_{[L:K]=2}\frac{1}{\Nm_K(\Disc(L/K))^s}.
\end{equation*}
It is proved in Theorem 1.1 of \cite{CDOQuartic} that
\begin{equation*}
  \Phi_{K,2}(s)=-1+ \frac{2^{-\complex(K)}}{\zeta_K(2s)}\cdot \sum_{\fc\mid
    2} \ \frac{\Nm_K(2/\fc)}{\Nm_K(2/\fc)^{2s}} \cdot \sum_{\chi \in \Cl(K,\fc^2)^\vee}L_K(s,\chi),
\end{equation*}
where $\complex(K)$ denotes the number of pairs of complex embeddings
of $K$, $\fc$ runs over all integral ideals of $K$ dividing $2$,
$\chi$ runs over all quadratic characters of the ray class group
modulo $\fc^2$, and $L_K(s,\chi)$ is the $L$-function of $K$ for
$\chi$.  It is also proven in Corollary 1.2 of \cite{CDOQuartic} that the
rightmost pole of $\Phi_{K,2}(s)$ is at $s=1$ with residue given by
\begin{equation}\label{constant}
  \Res_{s=1}\Phi_{K,2}(s)=
  \frac{2^{-\complex(K)}}{\zeta(2)}\cdot\frac{L(1,K/\bQ)}{L(2,K/\bQ)}.
\end{equation}
We can then obtain ``smooth counts'' of the number of 
quadratic extensions of quadratic fields $K$:
\begin{lemma}\label{lemquadcount}
Let $\varphi$ be a smooth compactly supported function
$\varphi:\R_{\geq 0}\to\R_{\geq 0}$. If $K$ is quadratic, then
\begin{equation*}
  \begin{array}{rcl}
\displaystyle\sum_{[L:K]=2}\varphi\Bigl(\frac{|\Disc(K)\cdot\Nm_K(\Disc(L/K)|}{X}\Bigr)
  &=&\displaystyle\Vol(\varphi)\cdot\Res_{s=1}\Phi_{K,2}(s)\cdot\frac{X}{|\Disc(K)|}\\[.2in]&&+ \ 
\displaystyle
O_{\epsilon,\varphi}(2^{\omega(\Disc(K))}|\Disc(K)|^{-1/4+\epsilon}X^{1/2+\epsilon}),
  \end{array}
  \end{equation*}
where $\omega(d)$ denotes the number of prime divisors of $d$, and
$\Vol(\varphi)$ denotes $\int\varphi(t)dt$.
\end{lemma}
\begin{proof}
Let $\tilde{\varphi}$ denote the Mellin transform of $\varphi$. By Mellin
inversion, we see that the left hand side of the above equation is
equal to
\begin{equation*}
\frac{1}{2\pi
  i}\int_{{\rm Re}(s)=2}
  \tilde{\varphi}(s)\cdot\frac{X^s}{|\Disc(K)|^s}\cdot\Phi_{K,2}(s)\ ds.
\end{equation*}
Shifting the line of integration to ${\rm Re}(s)=1/2+\epsilon$, we
pick up the main term from the pole at $1$ since
$\tilde{\varphi}(1)=\Vol(\varphi)$. The error term follows by using the
convexity bound of $O_\epsilon(|\Disc(K)|^{1/4+\epsilon})$ near $s=1/2$ at each of
the $O\left(2^{\omega(\Disc(K))}\right)$ $L$-functions used to define
$\Phi_{K,2}(s)$.
\end{proof}

\noindent The next lemma, whose proof is very similar to that of Lemma
\ref{lemquadcount}, gives a bound when imposing ramification on the quadratic extensions of a fixed quadratic field $K$ and will only be used in Section 8:

\begin{lemma}\label{lemunifcondp}
If $K$ is a quadratic field and $\ff$ is a squarefree product
of prime ideals in $K$, then the number of quadratic extensions $L$ over
$K$ such that $\Nm_K(\Disc(L/K))<X$ and every prime dividing $\ff$
ramifies in $L$ is bounded by
\begin{equation*}
  O_\epsilon\biggl(\frac{L(1,K/\bQ)}{\Nm(\ff)}\cdot X+\frac{|\Disc(K)|^{-1/4}}{\Nm(\ff)^{1/2}}\cdot X^{1/2+\epsilon}\biggr).
\end{equation*}
\end{lemma}
\begin{proof}
We can (and do) assume that $\ff$ is odd. Since an upper bound for the number of quadratic extensions of $K$ can
be obtained with a smooth sum, we proceed as in the proof of Lemma
\ref{lemquadcount}. The only difference is that we use, instead of
$\Phi_{K,2}(s)$, the Dirichlet series $\Phi_{K,2,\ff}(s)$
corresponding to extensions $L$ of $K$ that are ramified at every
prime dividing $\ff$:
  \begin{equation*}
    \begin{array}{rcl}
      \displaystyle\Phi_{K,2,\ff}(s)&\defeq&
      \displaystyle\sum_{\substack{[L:K]=2 \\ L\text{ ramified at $\ff$}}} \frac{1}{\Nm_K(\Disc(L/K))^s}\\\\&=&
      \displaystyle    -1 + \frac{2^{-\complex(K)}}{\zeta_K(2s)}  \cdot \sum_{\fc \mid 2}  \ \Nm_K(\fc)^{2s - 1}  \cdot     \sum_{\chi\in \Cl(K,\fc^2)^\vee}
     \biggl( \sum_{\substack{\fa\text{ squarefree} \\ (\fa, \fc) = 1, \ \ff \mid \fa}} \frac{\chi(\fa)}{\Nm_K(\fa)^s}\biggr),
    \end{array}
  \end{equation*}
  where the notation is as in the definition of $\Phi_{K,2}(s)$. Since
  the residue of $\Phi_{K,2,\ff}(s)$ at $1$, its rightmost pole, is
  $\ll L(1,K/\bQ)/\Nm(\ff)$, the lemma follows from an argument
  identical to the proof of Lemma~\ref{lemquadcount}.
\end{proof}

\subsection{ Proof of Theorem \ref{thm:analyticmain2}}
We are now ready to prove the main result of this section. From Lemma \ref{thm:abelianbounds}, it follows that we may estimate
$\NumC(\Sigma;X,X^\beta)$ by counting quadratic extensions $L$ of
quadratic fields $K$. Let $\chi_{[0,1]}$ denote the characteristic
function of $[0,1]$. Then
\begin{equation}\label{eqd4antemp}
\NumC(\Sigma;X,X^\beta) = \frac{1}{2}\sum_{\substack{K \in
    \KK(\Sigma)\\|\Disc(K)|<X^{\beta}}}\sum_{[L:K]=2}\chi_{[0,1]}\Bigl(\frac{|\Disc(K)\cdot\Nm_K(\Disc(L/K))|}{X}\Bigr).
\end{equation}
The factor of $1/2$ in the right hand side of \eqref{eqd4antemp} is to
account for the fact that a $D_4$-fields $L$ and its conjugate $L'$
both contribute to the inner sum, while the left hand side of
\eqref{eqd4antemp} counts $D_4$-fields up to conjugacy.

For $\epsilon>0$, choose $\varphi^\pm$ to be smooth compactly
supported functions such that $\varphi^\pm-\chi_{[0,1]}$ takes values
in $\R^\pm$ and such that $\Vol(\varphi^\pm)=1\pm\epsilon$. Lemma
\ref{lemquadcount} together with \eqref{constant} implies that
\begin{equation*}
\begin{array}{rcl}
  \displaystyle\sum_{\substack{K \in \KK(\Sigma)\\|\Disc(K)|<X^{\beta}}}\sum_{[L:K]=2}\varphi^{\pm}\Bigl(\frac{|\Disc(K)\cdot\Nm_K(\Disc(L/K))|}{X}\Bigr) &=&
  \!\!\!\!\displaystyle\sum_{\substack{K \in \KK(\Sigma)\\|\Disc(K)|<X^{\beta}}}
  \!\!\frac{1\pm \epsilon}{\zeta(2)}\cdot
  \frac{L(1,K/\bQ)}{L(2,K/\bQ)}\cdot\frac{2^{-\complex(K)}}{|\Disc(K)|}\cdot X
  \\[.35in]&+&
  \displaystyle O_{\epsilon}\Bigl(\sum_{\substack{K \in \KK(\Sigma)\\|\Disc(K)|<X^{\beta}}}
  |\Disc(K)|^{-\frac14+\epsilon}X^{\frac12+\epsilon}\Bigr).
\end{array}
\end{equation*}
In the above equation, the left hand side corresponding to $\varphi^+$
(resp.\ $\varphi^-$) is an upper bound (resp.\ lower bound) for
$\NumC(\Sigma;X,X^\beta)$. Meanwhile the error term on the right hand
side is bounded by $O_\epsilon(X^{\frac12+\frac{3\beta}{4}+\epsilon})$ which,
when $\beta<2/3$, is bounded by $o(X)$. Therefore, Theorem
\ref{thm:analyticmain2} follows by letting $\epsilon$ tend to $0$.
\hfill {$\Box$ \vspace{2 ex}}

\begin{rmk} We note that standard analytic methods (namely, Perron's
formula in conjunction with hybrid (sub)convexity bounds on the growth of
Hecke $L$-functions in the critical strip) yield a power saving in the
error bound in Theorem \ref{thm:analyticmain2}. However, we do not
include the arguments since they will not be necessary for the results
of this paper.
\end{rmk}

\section{Mass formulae for families of $D_4$-fields}

We now turn to the proof of Theorem \ref{2}. In the previous section,
the constant in the asymptotic number of $D_4$-fields with bounded
conductor whose quadratic subfield has small discriminant was determined as a sum of $L$-values. In
\S5.1 and \S5.2, we prove an identity relating the constant in Theorem \ref{thm:analyticmain2} to an Euler product matching the predicted mass formula described in \S3.3 by proving that the main contribution of the sum in the right hand side of Theorem \ref{thm:analyticmain2} comes from certain diagonal terms. Finally, in \S5.3, we
study the family of $D_4$-fields ordered by discriminant, and we prove an interesting identity between the analogous diagonal terms and the heuristic predicted by \eqref{eqDXY}.

\subsection{Isolating the diagonal terms in Theorem \ref{thm:analyticmain2}} 
We first prove a lemma that will be used in bounding the non-diagonal terms when we calculate the sum of $L$-values that appear in
Theorem \ref{thm:analyticmain2} in terms of a weighted M\"obius sum.

\begin{lem}\label{aralem}For any $\epsilon>0$,
\begin{equation*}\sum_{\substack{0 < D< X \\D \ \mbox{\scriptsize{squarefree}}}}\frac{1}{D} \cdot \Biggl(\sum_{m=1}^{\infty}\sum^{D^{\frac12+\epsilon}}_{\substack{n = 1\\ mn\neq \square}}\frac{\mu(m)}{m^2n}\left(\frac{D}{mn}\right)\Biggr)=O_{\epsilon}(1),
\end{equation*}
where $\left(\frac{\cdot}{\cdot}\right)$ denotes the Legendre symbol.
\end{lem}

\begin{proof}
The $m$-sum is absolutely convergent, 
so we will focus on the $n$ and the
$D$-sums. Interchanging the $n$ and $D$-sums yields
\begin{equation}\label{eqbound1}
\sum_{\substack{0 < D< X \\D \ \mbox{\scriptsize{squarefree}}}}\frac{1}{D}\cdot\sum_{\substack{n< D^{\frac12+\epsilon}\\ mn\neq \square}}\frac{1}{n}\left(\frac{D}{mn}\right)=\sum_{\substack{n< X^{\frac12+\epsilon}\\ mn\neq \square}}\frac1n \cdot\sum_{\substack{n^{\frac{2}{1+2\epsilon}}< D< X\\ D \ \mbox{\scriptsize{squarefree}}}}\frac1D\left(\frac{D}{mn}\right).
\end{equation}
We will now apply a simple squarefree sieve to complete the $D$-sum and then use the P\'olya-Vinogradov inequality to finish the estimate. In particular, we can rewrite \eqref{eqbound1} as
$$
\sum_{\substack{n< X^{\frac12+\epsilon}\\ mn\neq \square}}\frac1n \cdot \Biggl(\sum_{\alpha< n^{\frac{1}{1+2\epsilon}}}\biggl(\frac{\mu(\alpha)}{\alpha^2}\cdot\sum_{\substack{n^{\frac{2}{1+2\epsilon}}\leq \alpha^2 d< X}}\frac1d\cdot\Bigl(\frac{\alpha^2 d}{mn}\Bigr)\biggr)+\sum_{ n^{\frac{1}{1+2\epsilon}}\leq\alpha<X^{\frac12}}\biggl(\frac{\mu(\alpha)}{\alpha^2}\cdot\sum_{\substack{n^{\frac{2}{1+2\epsilon}}\leq \alpha^2 d< X}}\frac1d\cdot\Bigl(\frac{\alpha^2 d}{mn}\Bigr)\biggr)\Biggr).
$$
Thus, \eqref{eqbound1} is bounded by
\begin{align*}
&\ll\sum_{\substack{n< X^{\frac12+\epsilon}\\mn\neq \square}}\frac1n \cdot \biggl(\sum_{\alpha< n^{\frac{1}{1+2\epsilon}}} \frac{1}{\alpha^2}\cdot \Bigl|\sum_{\alpha^{-2}n^{\frac{2}{1+2\epsilon}} \leq  d < \alpha^{-2}X} \frac1d\cdot\left(\frac{\alpha^2 d}{mn}\right)\Bigr|\ +\sum_{ n^{\frac{1}{1+2\epsilon}}<\alpha<X^{\frac12}} \frac{1}{\alpha^2}\cdot\Bigl|\sum_{\substack{d<\frac{ X}{\alpha^2}}} \frac1d\cdot\left(\frac{\alpha^2 d}{mn}\right)\Bigr|\biggr)\\
&\ll\sum_{\substack{n< X^{\frac12+\epsilon}\\ mn\neq \square}}\frac{m^{\frac12}\log(n)}{n^{\frac12+\frac{1}{1+2\epsilon}}} \\
&=O_{\epsilon}(m^{\frac12}).
\end{align*}
The lemma then follows from the absolute convergence of $\sum m^{-\frac32}$.
\end{proof}

The next result is the key input in obtaining the mass formula. Using Lemma \ref{aralem}, we rewrite the sum of $L$-values appearing in Theorem \ref{thm:analyticmain2} in terms of a weighted M\"{o}bius sum that we will later show is equal to an Euler product. When ordering $D_4$-fields by discriminant, there is no known analogue to Proposition \ref{thm:analyticmain}.
\begin{prop}\label{thm:analyticmain} We have:
\begin{equation} \label{lemma41}
\begin{array}{c}
\displaystyle\sum_{\substack{[K:\mathbb{Q}]=2\\ 0 < \Disc(K)<  X}}\frac{L(1,K/\mathbb{Q})}{L(2,K/\mathbb{Q})}\cdot\frac{1}{|\Disc(K)|} \ \ = \displaystyle\sum_{\substack{[K:\mathbb{Q}]=2\\0< \Disc(K)<X}}\frac{1}{|\Disc(K)|}\cdot\sum_{\substack{0 < a,b < \infty \\ (\Disc(K),ab) = 1}}\frac{\mu(a)}{a^3b^2}\ +\ O(1);\\\\
\displaystyle\sum_{\substack{[K:\mathbb{Q}]=2\\ -X < \Disc(K)<  0}}\frac{L(1,K/\mathbb{Q})}{L(2,K/\mathbb{Q})}\cdot\frac{1}{|\Disc(K)|}  \ \  =\displaystyle\sum_{\substack{[K:\mathbb{Q}]=2\\-X< \Disc(K)<0}}\frac{1}{|\Disc(K)|}\cdot\sum_{\substack{0 < a,b < \infty \\ (\Disc(K),ab) = 1}}\frac{\mu(a)}{a^3b^2} \ + \ O(1).
\end{array}
\end{equation}
\end{prop}
\begin{proof}
Let $\chi_{K}$ denote the quadratic character associated with $K$ by class field theory so that we have $L(1,K/\bQ) = \sum_{n=1}^\infty \frac{\chi_K(n)}{n}$. From the absolutely convergent Euler product, it is straightforward to see that $L(2,K/\Q)>(\zeta(4)/\zeta(2))^2>0$, and hence $1/L(2,K/\Q)$ is uniformly bounded independent of $K$.
Using partial summation and the P\'olya-Vinogradov inequality, for any $\epsilon>0$ we get
\[\frac{1}{L(2,K/\bQ)}\cdot\sum_{\substack{n> |\Disc(K)|^{\frac12+\epsilon}}}\frac{\chi_{K}(n)}{n}=O_{\epsilon}\left(\frac{\log(|\Disc(K)|)}{|\Disc(K)|^{\epsilon}}\right).\]
Thus, we can conclude that
\begin{equation}\label{lemma43}
\frac{L(1,K/\bQ)}{L(2,K/\bQ)}=\frac{1}{L(2,K/\bQ)}\cdot\sum_{n=1}^{|\Disc(K)|^{\frac12+\epsilon}}\frac{\chi_{K}(n)}{n}+O_{\epsilon}\left(\frac{\log(|\Disc(K)|)}{|\Disc(K)|^{\epsilon}}\right).
\end{equation}
Using \eqref{lemma43}, the left hand sides of \eqref{lemma41} are equal to
\begin{equation}\label{lemma41'}
\begin{array}{c}
\displaystyle\sum_{\substack{[K:\mathbb{Q}]=2\\ 0 < \Disc(K) <X}}\frac{1}{|\Disc(K)|}\cdot\Biggl(\frac{1}{L(2,K/\bQ)}\cdot\sum_{n=1}^{  \Disc(K)^{\frac12+\epsilon}}\frac{\chi_{K}(n)}{n}+O_{\epsilon}\left(\frac{\log(\Disc(K))}{\Disc(K)^{\epsilon}}\right)\Biggr);\\
\displaystyle \sum_{\substack{[K:\mathbb{Q}]=2\\ -X < \Disc(K) <0}}\frac{1}{|\Disc(K)|}\cdot\Biggl(\frac{1}{L(2,K/\bQ)}\cdot\sum_{n=1}^{ |\Disc(K)|^{\frac12+\epsilon}}\frac{\chi_{K}(n)}{n}+O_{\epsilon}\left(\frac{\log(|\Disc(K)|)}{|\Disc(K)|^{\epsilon}}\right)\Biggr).
\end{array}
\end{equation}
In either case, the sum of the $O_\epsilon$ terms is itself $O_\epsilon(1)$, and so we focus on the remaining term. Using the absolute convergence of the Euler product of $L(2,K/\bQ)^{-1}$, we have
\begin{equation}\label{lemma44}
\frac{1}{L(2,K/\bQ)}\cdot\Biggl(\sum_{n=1}^{  |\Disc(K)|^{\frac12+\epsilon}}\frac{\chi_K(n)}{n}\Biggr) \ = \ \Biggl(\sum_{\substack{m=1}}^{\infty}\frac{\mu(m)\chi_{K}(m)}{m^2}\Biggr)\cdot\Biggl(\sum_{n=1}^{  |\Disc(K)|^{\frac12+\epsilon}}\frac{\chi_{K}(n)}{n}\Biggr).
\end{equation}
The key observation we make is that the main contribution to the right hand side of \eqref{lemma44} comes from the ``diagonal'' terms, i.e., when $mn$ is a square. By pulling out these terms, we may rewrite \eqref{lemma44} as
\begin{equation}\label{lemma45}
\sum_{\substack{0 < a,b < \infty \\ (\Disc(K),ab) = 1}}\frac{\mu(a)}{a^3b^2}+\sum_{\substack{n=1}}^{|\Disc(K)|^{\frac12+\epsilon}}\frac{\chi_{K}(n)}{n}\cdot\sum_{\substack{m=1\\ mn\neq \square}}^{\infty}\frac{\mu(m)\chi_{K}(m)}{m^2}.
\end{equation}
Substituting \eqref{lemma45} back into \eqref{lemma41'} implies that the left hand sides of \eqref{lemma41} are equal to
\begin{equation*}
\sum_{\substack{[K:\mathbb{Q}]=2\\ 0 < \Disc(K)<X}}\frac{1}{\Disc(K)}\cdot\Biggl(\sum_{{\substack{0 < a,b < \infty \\ (\Disc(K),ab) = 1}}}\frac{\mu(a)}{a^3b^2} \ +\sum_{\substack{n=1}}^{\Disc(K)^{\frac12+\epsilon}}\sum_{\substack{m=1\\ mn\neq \square}}^{\infty}\frac{\mu(m)\chi_{K}(mn)}{m^2n}\Biggr)+O_{\epsilon}(1);
\end{equation*}
\begin{equation*}
\sum_{\substack{[K:\mathbb{Q}]=2\\ -X < \Disc(K)<0}}\frac{1}{|\Disc(K)|}\Biggl(\sum_{{\substack{0 < a,b < \infty \\ (\Disc(K),ab) = 1}}}\frac{\mu(a)}{a^3b^2} \ +\sum_{\substack{n=1}}^{|\Disc(K)|^{\frac12+\epsilon}}\sum_{\substack{m=1\\ mn\neq \square}}^{\infty}\frac{\mu(m)\chi_{K}(mn)}{m^2n}\Biggr)+O_{\epsilon}(1).
\end{equation*}
By Lemma \ref{aralem},
$$\sum_{\substack{[K:\mathbb{Q}]=2\\ |\Disc(K)|<X}}\frac{1}{|\Disc(K)|}\cdot\biggl( \sum_{\substack{n=1}}^{|\Disc(K)|^{\frac12+\epsilon}}\sum_{\substack{m=1\\ mn\neq \square}}^{\infty}\frac{\mu(m)\chi_{K}(mn)}{m^2n}\biggr) = O_{\epsilon}(1).$$
Noting that the remaining term does not depend on $\epsilon$, we obtain the proposition.
\end{proof}

\subsection{Proof of Theorem \ref{2}}

We now turn to the proof of Theorem \ref{2}. From the identity
	\begin{equation}\label{identity}\sum_{\substack{0 < a,b < \infty \\ (\Disc(K),ab) = 1}}\frac{\mu(a)}{a^3b^2}=\frac{\zeta(2)}{\zeta(3)}\cdot\prod_{p\mid \Disc(K)}\frac{1-\frac{1}{p^2}}{1-\frac{1}{p^3}}, \end{equation}
we immediately obtain: \begin{equation}\label{prop53}
\begin{array}{c}
\displaystyle\sum_{\substack{[K:\mathbb{Q}]=2\\ 0<\Disc(K)<X}}\frac{1}{\Disc(K)}\cdot\sum_{\substack{0 < a,b < \infty \\ (\Disc(K),ab) = 1}}\frac{\mu(a)}{a^3b^2} \ = \ \displaystyle\frac{\zeta(2)}{\zeta(3)}\cdot\biggl(\sum_{\substack{[K:\mathbb{Q}]=2\\ 0<\Disc(K)<X}}\frac{1}{\Disc(K)}\cdot\prod_{p\mid \Disc(K)}\frac{1-\frac{1}{p^2}}{1-\frac{1}{p^3}}\biggr);\\
\displaystyle\sum_{\substack{[K:\mathbb{Q}]=2\\ -X<\Disc(K)<0}}\frac{1}{|\Disc(K)|}\cdot\sum_{\substack{0 < a,b < \infty \\ (\Disc(K),ab) = 1}}\frac{\mu(a)}{a^3b^2}  = \displaystyle\frac{\zeta(2)}{\zeta(3)}\cdot\biggl(\sum_{\substack{[K:\mathbb{Q}]=2\\ -X<\Disc(K)<0}}\frac{1}{|\Disc(K)|}\cdot\prod_{p\mid \Disc(K)}\frac{1-\frac{1}{p^2}}{1-\frac{1}{p^3}}\biggr).
\end{array}
\end{equation}
Decomposing the right hand sides of \eqref{prop53} into sums over squarefree integers in a fixed congruence class $\bmod \,\,4,$ we obtain that the left hand sides of \eqref{prop53} are equal to
\begin{equation}\label{prop53b}
\begin{array}{c}
\displaystyle\frac{\zeta(2)}{\zeta(3)}\cdot\Biggl(\sum_{\substack{ 1<D<X\\ D\equiv 1 \bmod 4\\ D \mbox{\scriptsize{ squarefree}}}} \frac1D\cdot\prod_{p\mid D}\frac{1-\frac{1}{p^2}}{1-\frac{1}{p^3}}+\sum_{\substack{ 1< D<X\\ D\equiv 3 \bmod 4\\ D \mbox{\scriptsize{ squarefree}}}}\frac{3}{14D}\cdot\prod_{p\mid D}\frac{1-\frac{1}{p^2}}{1-\frac{1}{p^3}}+\sum_{\substack{ 1< D<X\\ D\equiv 1\bmod 2\\  D \mbox{\scriptsize{ squarefree}}}} \frac{3}{28D}\cdot\prod_{p\mid D}\frac{1-\frac{1}{p^2}}{1-\frac{1}{p^3}}\Biggr);\\\\
\displaystyle\frac{\zeta(2)}{\zeta(3)}\cdot\Biggl(\sum_{\substack{ -X<D<-1\\ D\equiv 1 \bmod 4\\ D \mbox{\scriptsize{ squarefree}}}} \frac1D\cdot\prod_{p\mid D}\frac{1-\frac{1}{p^2}}{1-\frac{1}{p^3}}+\sum_{\substack{ -X< D<-1\\ D\equiv 3 \bmod 4\\ D \mbox{\scriptsize{ squarefree}}}}\frac{3}{14D}\cdot\prod_{p\mid D}\frac{1-\frac{1}{p^2}}{1-\frac{1}{p^3}}+\sum_{\substack{ -X< D<-1\\ D\equiv 1\bmod 2\\  D \mbox{\scriptsize{ squarefree}}}} \frac{3}{28D}\cdot\prod_{p\mid D}\frac{1-\frac{1}{p^2}}{1-\frac{1}{p^3}}\Biggr).
\end{array}
\end{equation}
Consider the limit
\begin{equation*}
\begin{array}{rcl}
\displaystyle\lim_{X\to\infty}\frac{\zeta(2)}{\zeta(3)\log (X)}
\cdot \sum_{\substack{1<D<X\\D\mbox{\scriptsize{ squarefree}}}}\frac{1}{D}\cdot\prod_{p\mid D}
\frac{1-\frac{1}{p^2}}{1-\frac{1}{p^3}}&=&
\displaystyle\lim_{X\to\infty}\frac{\zeta(2)}{\zeta(3)\log (X)}\cdot
\sum_{\substack{1<D<X\\D\mbox{\scriptsize{ squarefree}}}}\prod_{p\mid D}
\frac1{p}\cdot\frac{1-\frac{1}{p^2}}{1-\frac{1}{p^3}}\\[.4in]&=&
\displaystyle\frac{\zeta(2)}{\zeta(3)}\cdot\prod_p\Bigl(
1+\frac1{p}\cdot\frac{1-\frac{1}{p^2}}{1-\frac{1}{p^3}}\Bigr)
\Bigl(1-\frac{1}{p}\Bigr)
\\[.25in]&=&
\displaystyle\zeta(2)\cdot\prod_p\Bigl(1+\frac{1}{p}-\frac{2}{p^3}\Bigr)
\Bigl(1-\frac{1}{p}\Bigr)
\\[.25in]&=&
\displaystyle\zeta(2)\cdot\prod_p\Bigl(1-\frac{1}{p^2}-\frac{2}{p^3}+\frac{2}{p^4}
\Bigr).
\end{array}
\end{equation*}
Carrying out the analogous computation for each term in both
equations of \eqref{prop53b} yields Theorem~\ref{2}. \hfill
{$\Box$ \vspace{2 ex}}

Theorems \ref{2} and \ref{thm:analyticmain2}
immediately imply the following result.
\begin{thm}\label{5.3}
  Let $\beta<2/3$ be fixed. For the family of all $D_4$-fields,
  \begin{equation*}
    N_{\CC}(X,X^\beta)=\frac{3\beta}{8}\cdot
    \prod_{p}\biggl(1-\frac{1}{p^2}-\frac{2}{p^3}+\frac{2}{p^4}\biggr)\cdot X\log (X)
    +O(X).
  \end{equation*}
\end{thm}
\noindent It follows from the heuristics of \S3.3 that the family of $D_4$-fields
$L$ satisfying $|\D(L)| \leq |\CC(L)|^{2/3}$
satisfies the mass formula \eqref{eqCXY} when such fields are ordered by their conductors. In Sections 6-8, we prove a refinement of Theorem \ref{5.3} by adapting the arguments in \cite{BhargavaQuarticCount} in conjunction with the analytic techniques used in Section 4.

\subsection{Diagonal terms for the family of $D_4$-fields ordered by discriminant}

We would like to conclude this section by considering asymptotics for the analogous ``diagonal terms'' that arise when counting
the number of $D_4$-fields having bounded discriminants.  Let $\NumDD(X)$
denote the number of isomorphism classes of $D_4$-fields $L$ with $|\disc(L)|<X$. By
Corollary 1.4 of \cite{CDOQuartic} (or following the proof of Theorem
\ref{thm:analyticmain2} with discriminant in place of conductor), we
have
\begin{equation}\label{disceq}
\NumDD(X)=\frac{X}{2\zeta(2)}\cdot\sum_{\substack{[K:\Q]=2 \\ |\disc(K)| < X}}\frac{L(1,K/\Q)}{L(2,K/\Q)}\cdot\frac{2^{-\complex(K)}}{|\disc(K)|^2}+o(X).
\end{equation}
Applying the reasoning in the proof of Proposition \ref{thm:analyticmain}, we obtain that combining Equation \eqref{lemma44} and \eqref{disceq} implies
\begin{equation*}
\NumDD(X)=\frac{X}{2\zeta(2)}\cdot\sum_{{\substack{[K:\Q]=2 \\ |\disc(K)| < X}}}\biggl(\sum_{\substack{m=1}}^{\infty}\frac{\mu(m)\chi_{K}(m)}{m^2}\biggr)\cdot\biggl(\sum_{n=1}^{  \infty}\frac{\chi_{K}(n)}{n}\biggr)\cdot\frac{2^{-\complex(K)}}{|\disc(K)|^2}+o(X).
\end{equation*}
We replace the product of sums 
$$\biggl(\sum_{\substack{m=1}}^{\infty}\frac{\mu(m)\chi_{K}(m)}{m^2}\biggr)\cdot \biggl(\sum_{n=1}^{
  \infty}\frac{\chi_{K}(n)}{n}\biggr),$$
 with simply the ``diagonal'' terms (i.e., the terms where $mn$ is a square).

 \begin{thm}\label{5.4} We have the following:
\begin{equation}\label{disc4}
\frac{1}{2\zeta(2)}\cdot\sum_{\substack{[K:\Q]=2 \\ |\Disc(K)| < X}}\frac{2^{-r(K)}}{|\Disc(K)|^2}\cdot\biggl(\sum_{\substack{0 < a,b < \infty \\ (\Disc(K),ab) = 1}}\frac{\mu(a)}{a^3b^2}\biggr) \cdot X \ \ \sim \ \ \frac{3\cdot11^2}{2^6\cdot 17}\cdot\prod_p\left(1+\frac{1}{p^2}-\frac{1}{p^3}-\frac{1}{p^4}\right) \cdot X.
\end{equation}
\end{thm}
\begin{proof} This follows from an argument analogous to the proof
of Theorem \ref{2}.
\end{proof}
This term agrees exactly with the heuristic in \eqref{eqDXY}!  The non-diagonal terms, as in \S5.2, again give an error term of $O(X)$. In
the case when $D_4$-fields were ordered by conductor, this error term
was negligible compared to the main term of $\asymp X\log X$. This
time, however, the main term of $\asymp X$ does not automatically
dominate the error term. In fact, the comparison of the constant $c \approx 0.0523$ from \cite{CDOQuartic} and the constant $\approx 0.406$ on the right hand side of \eqref{disc4} implies that the non-diagonal terms do make
a non-negligible contribution.

\section{Parameterizing $D_4$-fields via pairs of ternary quadratic forms}
\label{sec:parameterization}

We next give a proof using geometry-of-numbers techniques in conjunction with arithmetic invariant theory methods for determining asymptotics on the number of $D_4$-fields with $|\q| < X$ and $|\D| \ll X$. We obtain worse error estimates in this second proof, but we are able to prove more refined statements for a wider class of collections of local specifications. We begin with a parametrization of $D_4$-fields via certain pairs of ternary quadratic forms, following Bhargava \cite{BhargavaQuarticComposition} and Wood \cite{WoodThesis}. In \S6.1, we describe the arithmetic invariant theory for orbits of such pairs of ternary quadratic forms and compare it to the invariants defined in Definitions \ref{fundinv} and \ref{condD4} for the corresponding $D_4$-fields. We additionally define splitting types, and we compute the $p$-adic densities for pairs of ternary quadratic forms corresponding to $D_4$-fields with fixed splitting type at $p$. These results allow us to employ geometry-of-numbers methods carried out in Section 7 to count the relevant orbits parametrizing $D_4$-fields with $|\q| < X$ and $|\D| \ll X$. 

\medskip

In \cite{BhargavaQuarticComposition}, Bhargava proved that
isomorphism classes of pairs $(Q,C)$, where $Q$ is a quartic ring
and $C$ is a cubic resolvent ring of $Q$ are in bijection with
$\GL_2(\Z)\times\SL_3(\Z)$-orbits on $\Z^2\otimes\Sym^2(\Z^3)$,
the space of pairs of integral ternary quadratic forms. If $Q$ is a
maximal quartic ring, then it has a unique cubic resolvent ring, so
this bijection (when restricted to maximal rings) can be viewed as a
parametrization of quartic fields. We write a pair of ternary
quadratic forms as a pair of symmetric $3\times 3$ matrices $(A,B)$ whose
diagonal entries are integers and nondiagonal entries are
half-integers. The group $\GL_2\times\SL_3$ acts on pairs of ternary
quadratic forms as follows:
\begin{equation*}
(g_2,g_3)\cdot(A,B)=(g_3Ag_3^t,g_3Bg_3^t)\cdot g_2^t.
\end{equation*}

For quartic rings $Q$ containing a quadratic subring, Wood
\cite{WoodThesis} gives a more specialized bijection: For any ring
$R$, let $V'(R)\subset R^2\otimes\Sym^2(R^3)$ denote the space of
pairs of ternary quadratic forms $(A,B)$ satisfying
\begin{equation*}
(A,B)=\left(
\begin{bmatrix}
0&0&0\\
0&a_{22}&\frac{a_{23}}{2}\\
0&\frac{a_{23}}{2}&a_{33}
\end{bmatrix},
\begin{bmatrix}
b_{11}&\frac{b_{12}}{2}&\frac{b_{13}}{2}\\
\frac{b_{12}}{2}&b_{22}&\frac{b_{23}}{2}\\
\frac{b_{13}}{2}&\frac{b_{23}}{2}&b_{33}
\end{bmatrix}\right),
\end{equation*}
where $a_{22}$, $a_{23}$, $a_{33}$, $b_{11}$, $b_{12}$, $b_{22}$,
$b_{23}$, and $b_{33}$ are elements of $R$ with $b_{11}\neq 0$. The
subgroup $G'(R)$ of $\GL_2(R) \times \SL_3(R)$ consisting of
elements $(g_2, g_3)$ such that 
\[ g_2 = \begin{bmatrix} \pm1 & 0 \\ \ast & \pm1 \end{bmatrix},
   \hspace{1cm} \mathrm{and} \hspace{1cm}
   g_3 = \begin{bmatrix} \pm1 & 0 & 0 \\ \ast & \ast & \ast \\ \ast & \ast & \ast \end{bmatrix},
\]
acts on $V'(R)$. Then the $G'(\Z)$-orbits on $V'(\Z)$ are in
bijection with triples $(Q,C,T)$ consisting of a quartic ring $Q$, a
cubic resolvent $C$ of $Q$, and a quadratic subring $T \subset
Q$. More precisely:

\begin{thm}[Thm.~7.3.5 of \cite{WoodThesis}]\label{thm:d4parameterization}
 For any principal ideal domain $R$ of characteristic different from 2, there is a bijection between $G'(R)$-equivalence classes of elements
  of $V'(R)$ with isomorphism classes of triples $(Q, C, T)$ where
  \begin{itemize}
    \item $Q$ is a quartic ring over $R$,
    \item $C$ is a cubic resolvent of $Q$ with map $\resolventmap: Q \to C$, and
    \item $T \subset Q$ is a primitive quadratic subalgebra such
      that $\resolventmap(T) \ne 0$.
  \end{itemize}
\end{thm}
 
\noindent In order to obtain a parametrization of {\it maximal} orders in
$D_4$-fields, we first make a few definitions. 
\begin{defn} An element of $v\in
V'(\Z)$ is {\em generic} if the quartic ring corresponding to $v$
under Theorem \ref{thm:d4parameterization} is a {\em $D_4$-order}, i.e., an
order in a $D_4$-field. Additionally, an element $V'(\Z)$
is said to be {\em maximal} if it corresponds to a maximal quartic
ring. 
\end{defn} Let $(A,B)$ be an element of $V'(\Z)$ and let $Q$ be the
quartic ring corresponding to it. It follows from Lemma 22 of \cite{BhargavaQuarticComposition}, that $Q$ is nonmaximal
at every prime dividing $b_{11}$. Hence $Q$ is maximal only when
$b_{11}=\pm 1$. Furthermore, by replacing $(A,B)$ with
$(-\Id,\Id)\cdot(A,B)=(-A,-B)$, if necessary, we may assume that
$b_{11}=1$. We define $V(\Z) \subset V'(\Z)$ to be the subspace of
pairs
\begin{equation}\label{v}
(A,B)=\left(
\begin{bmatrix}
0&0&0\\
0&a_{22}&\frac{a_{23}}{2}\\
0&\frac{a_{23}}{2}&a_{33}
\end{bmatrix},
\begin{bmatrix}
1&\frac{b_{12}}{2}&\frac{b_{13}}{2}\\
\frac{b_{12}}{2}&b_{22}&\frac{b_{23}}{2}\\
\frac{b_{13}}{2}&\frac{b_{23}}{2}&b_{33}
\end{bmatrix}\right),
\end{equation}
and we define the
subgroup $G(\Z) \subset G'(\Z)$ to be the set of pairs $(g_2,g_3) \in
G'(\Z)$ such that 
\begin{equation*} g_2 = \begin{bmatrix} \pm1 & 0 \\ \ast & 1 \end{bmatrix},
   \hspace{1cm} \mathrm{and} \hspace{1cm}
   g_3 = \begin{bmatrix} \pm1 & 0 & 0 \\ \ast & \ast & \ast \\ \ast & \ast & \ast \end{bmatrix}.
\end{equation*} Moreover, for any ring $R$, we analogously define the space $V(R)$
and the group $G(R)$. We have the following proposition:

\begin{prop}\label{propparam}
There is a bijection between $($isomorphism classes of$)$ $D_4$-fields
and maximal generic $G(\Z)$-orbits on $V(\Z)$.
\end{prop}
\begin{proof} Proposition \ref{propparam} follows immediately from Theorem
\ref{thm:d4parameterization} and the above discussion, in conjunction
with the fact that a maximal $D_4$-order has a unique resolvent ring
and primitive quadratic subalgebra. \end{proof}

\subsection{Invariant theory}
We next discuss the invariant theory for the action of $G$ on $V$. The action of $G(\C) \cap \SL_3(\C)$ on $V(\C)$ turns out to have ring of invariants freely generated by 3 elements. We can describe these invariants in terms of the {\em cubic resolvent} of $(A,B)$, i.e., the binary cubic form $\Det(Ax+By)$. It is straightforward to check that if $(A,B)$ is as in \eqref{v}, then the coefficient of $x^3$ in $\Det(Ax + By)$ is equal to zero. For a ring $R$, let $U(R)$ denote the subspace of space of binary
cubic forms consisting of elements
\begin{equation*}
f(x,y)=bx^2y+cxy^2+dy^3,
\end{equation*}
where $b$, $c$, and $d$ are elements of $R$.  Define $N_1$ to be
the group of lower triangular $2\times 2$-matrices with top left entry
$\pm1$ and bottom right entry $1$.  Then $N_1$ acts on $U$ via the
action 
	\begin{equation*} g\cdot f(x,y)=f((x,y)\cdot g^t).\end{equation*} Additionally, if $\pi:V(R)\to U(R)$ denotes the resolvent map $(A,B) \mapsto 4\Det(Ax+By)$, then for $(g_2,g_3)\in G(R)$,
\begin{equation*}
\pi((g_2,g_3)\cdot(A,B))=g_2\cdot\pi(A,B).\end{equation*} The
coefficients $b$, $c$, and $d$ of the binary cubic form $\pi(A,B)$ are
the invariants for the action of $G(\C)\cap \SL_3(\C)$ on $V(\C)$.
Therefore, the ring of invariants for the action of all of $G(\C)$
on $V(\C)$ is the same as the ring of invariants for the action of
$N_1(\C)$ on $U(\C)$. The latter ring is freely generated by two
elements $\D$ and $\q$, which can be associated to an element of $V$
as follows:
\begin{defn}\label{eqinv}If $(A,B)\in V(R)$ has resolvent
$f(x,y)=bx^2y+cxy^2+dy^3$, then
\begin{equation*}
\begin{array}{rcccl}
\D(A,B)&\defeq&\D(f)&\defeq&-b\\
\q(A,B)&\defeq&\q(f)&\defeq&c^2-4bd.
\end{array}
\end{equation*}
Additionally, we set
$$\CC(A,B) \ \defeq \ \D(A,B) \cdot\q(A,B)\;\;\mbox{   and   }\;\;
\Disc(A,B) \ \defeq\ \D(A,B)^2\cdot\q(A,B).$$
\end{defn}

We now relate the invariants $\D$
and $\q$ of an element $(A,B) \in V(\bZ)$ with the invariants of the quartic ring 
corresponding to the $G(\bZ)$-orbit of $(A,B)$ in Theorem \ref{thm:d4parameterization}.
\begin{prop}\label{propparaminv}
Let $(A,B)\in V(\Z)$ be an element with nonzero invariants $\D(A,B)$ and
$\q(A,B)$. Let $Q$ denote the quartic ring corresponding to $(A,B)$ and let
$T$ be its quadratic subalgebra. Then we have $\Disc(T)=\D(A,B)$
and $\Nm_T(\Disc(Q/T))=|\q(A,B)|$.
\end{prop} 
\begin{proof} Let $(A,B)$ be as in \eqref{v} with resolvent $f(x,y) = bx^2y + cxy^2 + dy^3$, and assume that $\D(A,B)$ and $\q(A,B)$ are nonzero. As described in Section 3 of \cite{BhargavaQuarticComposition}, one can describe the multiplicative structure on a (normal) $\bZ$-basis of $Q$ using the matrix coefficients of $(A,B)$. Indeed, if $\langle 1, \alpha_1, \alpha_2,\alpha_3\rangle$ is a $\bZ$-basis for $Q$, we can write its multiplication table as 
	$$\alpha_{i}\alpha_{j} = c_{ij}^0 + c_{ij}^1\alpha_1 + c_{ij}^2\alpha_2 + c_{ij}^3\alpha_3,$$
where $c_{ij}^k \in \bZ$ for $1\leq i,j\leq 3$ and $k \in \{0,1,2,3\}$ are determined completely by $a_{ij}$ and $b_{ij}$.
Equations 20, 21, and 22 of \cite{BhargavaQuarticComposition} with $a_{11} = a_{12} = a_{13} = 0$ and $b_{11} = 1$ imply that $c_{11}^2 = c_{11}^3 = 0$. Additionally, we obtain $c_{11}^0 = -a_{33}a_{22}$ and $c_{11}^1 = a_{23}$, i.e.,
	$$\alpha_1^2 = -a_{33}a_{22} + a_{23}\alpha_1,$$
and so $\bZ[\alpha_1] = \langle 1,\alpha_1 \rangle$ is a quadratic subalgebra of $Q$. By the description of $T$ given in the proof of Corollary 7.2.2 of \cite{WoodThesis}, we have that $T = \bZ[\alpha_1]$, and 
	$$\Disc(T) = \Disc(\alpha_1) = a_{23}^2 - 4a_{33}a_{22}.$$
Using the resolvent map $\pi(A,B) = 4\Det(Ax+By) = bx^2y
+ cxy^2 + dy^3$, we deduce that
	$$b = 4a_{22}a_{33} - a_{23}^2.$$ Thus, $\Disc(T) =
\D(A,B)$, as necessary.

Furthermore, by Proposition 10 of \cite{BhargavaQuarticComposition} and
since $\Disc(A,B) = \Disc(\pi(A,B))$, we have
	$$\Disc(Q) = \Disc(A,B) = \Disc(bx^2y + cxy^2 + dy^3) = b^2c^2 - 4b^3d.$$
The relative discriminant formula implies that
	$|\Disc(Q)| = |\Disc(T)^2\cdot\Nm_T(\Disc(Q/T))|$, and so we conclude that 
	\begin{equation*}
    \Nm_T(\Disc(Q/T)) = \left|\frac{\Disc(Q)}{\Disc(T)^2}\right| = |c^2 - 4bd|.
  \end{equation*}
The proposition follows.
\end{proof}

We slightly generalize the notion of conductor given in the introduction for \'etale quartic algebras over $R$. For a pair $(Q,T)$, where $Q$ is a \'etale quartic algebra over $R$ and $T$ is a
primitive quadratic subalgebra of $Q$, we set $\CC(Q,T) \defeq \Disc(Q)/\Disc(T).$ The following lemma will be useful in obtaining
a bound on the number of $G(\bZ)$-orbits that correspond to
non-maximal $D_4$-orders in Section 8.

\begin{lem}\label{lemsqfull}
If $(A,B) \in V(\bZ_p)$ corresponds to a non-maximal quartic order $Q$
contained in a degree 4 \'{e}tale extension $L_p$ of $\Q_p$ and a
quartic subalgebra $T$ contained in a quadratic subextension $K_p$ of
$L_p$, then $$p^2 \mid \frac{\CC(A,B)}{\CC(L_p,K_p)}.$$
\end{lem}
\begin{proof}
Since the index of $Q$ in the maximal order of $L$ is a multiple of
$p$, it follows that the discriminant of $(A,B)$ differs from the
discriminant of $L$ by a factor of at least $p^2$. Since
$\CC(A,B)=\Disc(A,B)/\Disc(T)$ and $\CC(L_p,K_p)=\Disc(L_p)/\Disc(K_p)$, the
lemma follows unless $\Disc(T)/\Disc(K_p)$ is divisible by $p$, or
equivalently, unless $T$ is not maximal in the ring of integers of $K_p$.

Assume that $T$ is not maximal in the ring of integers of $K_p$, and
thus has discriminant divisible by $p^2$. We know from Proposition
\ref{propparaminv} that $-\Disc(T)$ is the discriminant of the
quadratic form $a_{22}y^2+a_{23}yz+a_{33}z^2$ corresponding to
$A$. Hence, either $A$ is a multiple of $p$ or, after a change of
variables, we may assume that $p^2\mid a_{22}$ and $p\mid a_{23}$. In
the first case, the pair of rings $(Q_1,T_1)$ corresponding to
$(A/p,B)\in V(\Z)$ are over-orders of $Q$ and $T$ such that $p^4\mid
\CC(Q,T)/\CC(Q_1,T_1)$. In the second case, the argument is similar:
we use the pair $(A_2,B_2)$ obtained by multiplying the third row and
column of $A$ and $B$ by $p$, and dividing $A$ by $p^2$, which yield a
pair $(Q_2,T_2)$ of over-orders of $(Q,T)$ such that $p^2\mid
\CC(Q,T)/\CC(Q_1,T_1)$. The lemma follows.
\end{proof}

\subsection{Splitting types of pairs of ternary quadratic forms}
Let $p$ be a fixed prime. We say that a pair $(A,B)\in V(\F_p)$ is
{\it nondegenerate} if the zero sets in $\bP^2(\overline{\F}_p)$ of
the two ternary quadratic forms corresponding to $A$ and $B$ intersect
at four points counted with multiplicity. For nondegenerate elements
$(A,B)\in V(\F_p)$ such that $A$ is nonzero, we define the {\it
  quartic splitting type} at $p$ to be
\begin{equation*}
\varsigma_p(A,B)=(f_1^{e_1}f_2^{e_2}\cdots),
\end{equation*}
where the $f_i$'s are the degrees over $\F_p$ of the field of
definition of these points, and the $e_i$'s are their
multiplicities. Furthermore, recall that the top row and column of $A$
is $0$, and so in the notation of \eqref{v}, $A$ corresponds to a
quadratic form $g(y,z)=a_{22}y^2+a_{23}yz+a_{33}z^2$.  We define the
{\it quadratic splitting type} $\varsigma_p'(A,B)$ of $(A,B)$ to be
$(11)$ if $g(x,y)$ has two distinct roots in $\bP^1(\F_p)$, to be
$(2)$ if $g(x,y)$ has a pair of conjugate roots defined over a
quadratic extension of $\F_p$, and $(1^2)$ if $g(x,y)$ has a double
root. We then say that the pair
$(\sigmatwo_p(A,B),\sigmafour_p'(A,B))$ is the {\it splitting type} of
$(A,B)$ at $p$.

If $(A,B)$ is an element in $V(\Z)$ or $V(\Z_p)$, we define the
splitting type of $(A,B)$ at $p$ to be the splitting type of the
reduction modulo $p$ of $(A,B)$, assuming it is nondegenerate.  Let
$Q$ be the quartic ring corresponding to $(A,B)$, and let $T$ denote
the quadratic subring of $Q$ arising from Theorem
\ref{thm:d4parameterization}.  It follows from \S4.1 of
\cite{BhargavaQuarticComposition}, that the quartic splitting type of
$(A,B)$ is equal to the splitting type of $Q$. We have seen that the
quadratic subring $T$ of $Q$ corresponding to the pair $(A,B)$ is the
quadratic ring whose discriminant is the same as that of the binary
quadratic form corresponding to $A$. Hence, the quadratic splitting
type of $(A,B)$ is the same as the splitting type of $T$.

Given a pair $(L_p,K_p)$ of extensions of $\Q_p$, whose rings of
integers correspond to a pair $(A,B)\in V(\Z_p)$, we define
$\varsigma_p(L_p,K_p)$ to equal the splitting type of $(A,B)$.
Additionally, there are four possible splitting types $\varsigma =
(\varsigma_\infty,\varsigma_\infty')$ at $\infty$ for an element in
$V(\R)$ having nonzero discriminant.  The invariants $(\D,\q)$ of an
element $v\in V(\R)^{(\varsigma)}$ are constrained in the following
way:
\begin{equation*}
  \begin{array}{lcl}
    \varsigma=((1111),(11))&\Rightarrow& \q>0,\;\D>0;\\[.05in]
    \varsigma=((112),(11))&\Rightarrow& \q<0,\;\D>0;\\[.05in]
    \varsigma=((22),(11))&\Rightarrow& \q>0,\;\D>0;\\[.05in]
    \varsigma=((22),(2))&\Rightarrow& \q>0,\;\D<0.\\[.05in]
  \end{array}
\end{equation*}
We denote the set of elements in $V(\R)$ having
splitting type $\varsigma$ by $V(\R)^{(\varsigma)}$ and set
$V(\Z)^{(\varsigma)}=V(\Z)\cap V(\R)^{(\varsigma)}$.

\subsection{The density of maximal elements}
For a prime $p$ and splitting type $(\varsigma_p,\varsigma_p')$, let $T_p(\varsigma_p,\varsigma_p')$ denote
the set of elements $(A,B)\in V(\Z_p)$ whose splitting type at $p$ is
$(\varsigma_p,\varsigma_p')$ and let $\maximalatp(\varsigma_p,\varsigma_p')$ denote the set of elements $(A,B)\in
T_p(\varsigma_p,\varsigma_p')$ that correspond to quartic rings under Theorem \ref{thm:d4parameterization} that are maximal at $p$.
Identifying $V(\Z_p) \cong \Z_p^8$ by regarding the non-fixed entries of \eqref{v} as a vector, let $\mu$ denote the Haar measure normalized so that $V(\Z_p)$  have volume $1$. We have
the following result which computes the volumes of the sets
$\maximalatp(\varsigma_p,\varsigma_p')$.

\begin{prop}\label{propdenmax}
  We have
  \begin{eqnarray*}
    \mu(\maximalatp((1111),(11)))&=&{\scriptstyle\frac{1}{8}}(p-1)^3(p+1)/p^4\\
    \mu(\maximalatp((22),(11)))&=&{\scriptstyle\frac{1}{8}}(p-1)^3(p+1)/p^4\\
    \mu(\maximalatp((22),(2)))&=&{\scriptstyle\frac{1}{4}}(p-1)^3(p+1)/p^4\\
    \mu(\maximalatp((112),(11)))&=&{\scriptstyle\frac{1}{4}}(p-1)^3(p+1)/p^4\\
    \mu(\maximalatp((4),(2)))&=&{\scriptstyle\frac{1}{4}}(p-1)^3(p+1)/p^4\\
    \mu(\maximalatp((1^211),(11)))&=&{\scriptstyle\frac{1}{2}}(p-1)^3(p+1)/p^5\\
    \mu(\maximalatp((1^22),(11)))&=&{\scriptstyle\frac{1}{2}}(p-1)^3(p+1)/p^5\\
    \mu(\maximalatp((1^21^2),(1^2)))&=&{\scriptstyle\frac{1}{2}}(p-1)^3(p+1)/p^5\\
    \mu(\maximalatp((2^2),(1^2)))&=&{\scriptstyle\frac{1}{2}}(p-1)^3(p+1)/p^5\\
    \mu(\maximalatp((1^4)(1^2)))&=&\;\;(p-1)^3(p+1)/p^6\\
    \mu(\maximalatp((1^21^2),(11)))&=&{\scriptstyle\frac{1}{2}}(p-1)^3(p+1)/p^6\\
    \mu(\maximalatp((2^2),(2)))&=&{\scriptstyle\frac{1}{2}}(p-1)^3(p+1)/p^6
  \end{eqnarray*}
\end{prop}
\begin{proof}
  First note that the splitting type of $v\in V(\Z_p)$ depends only on
  the reduction of $v$ modulo $p$. It follows that the densities of
  the sets $T_p(\varsigma_p,\varsigma_p')$ can be computed by counting
  elements in $V(\F_p)$. The conditions that ensure the maximality of
  $v$ are listed in \S4.2 of \cite{BhargavaQuarticComposition}. We prove
  Proposition \ref{propdenmax} by computing the densities of
  $T_p(\varsigma_p,\varsigma_p')$ and then, for each $\varsigma_p$,
  determining the probability that $v\in
  T_p(\varsigma_p,\varsigma_p')$ is maximal.

    Let $(A,B)$ be an element of $V(\Z_p)$ having quadratic splitting
    type $(11)$. It follows that the quadratic form corresponding to $A$
    has two distinct roots in $\bP^1(\F_p)$. The number of
    possibilities for $\bar{A}$, the reduction of $A$ modulo $p$, is
    thus equal to $(p+1)p(p-1)/2$ giving a density of
    $(p+1)p(p-1)/(2p^3)$ for the possibilities of $A$. By a change of
    variables, we may assume that $(A,B)$ is of the form
  \begin{equation}\label{eqsp11}
    (\bar{A},\bar{B})=\left(\left[\begin{array}{ccc} 0 & 0 & 0 \\ 0 & 0 &
        1/2 \\ 0 & 1/2 & 0 \end{array} \right],
    \left[ \begin{array}{ccc} 1 & 0 & 0 \\ 0 & s & 0 \\ 0 &
        0 & t \end{array} \right] \right),
  \end{equation}
  when $p$ is odd and 
  \begin{equation}\label{eqtempden2}
    (\bar{A},\bar{B})=\left(\left[\begin{array}{ccc} 0 & 0 & 0 \\ 0 & 0 &
        1/2 \\ 0 & 1/2 & 0 \end{array} \right],
    \left[ \begin{array}{ccc} 1 & \alpha/2 & \beta/2
        \\ \alpha/2 & s & 0 \\ \beta/2 &
        0 & t \end{array} \right] \right),
  \end{equation}
when $p=2$.  The quartic splitting type of $(A,B)$ then has six
options: $(1111)$, $(112)$, $(22)$, $(1^211)$, $(1^22)$, and
$(1^21^2)$. When $p$ is odd, this splitting type depends on whether
$s$ and $t$ are residues modulo $p$, nonresidues modulo $p$, or
$0$. Their relative densities are easy to compute. To obtain the
splitting type $(1111)$, both $s$ and $t$ must be residues modulo $p$
which occurs with relative density $(p-1)^2/4p^2$. It is easy to check
that for $p=2$, the relative density of elements with quartic
splitting type $(1111)$ is again $1/16=(p-1)^2/4p^2$. Multiplying with
the density $(p^2-1)/2p^2$ of split quadratic forms arising from $A$
yields the density of $T_p((1111),(11))$. Since every element with
unramified splitting type is automatically maximal, it follows that
$T_p((1111),(11))=\maximalatp((1111),(11))$, yielding the first part of the
proposition.

We now compute the density of $\maximalatp((1^21^2),(11))$. Again, we may
assume that $(\bar{A},\bar{B})$ is of the form \eqref{eqsp11} when $p$
is odd and of the form \eqref{eqtempden2} when $p=2$. When $p$ is odd,
such a pair $(A,B)$ has quartic splitting type $(1^21^2)$ when
$s=t=0$, and when $p=2$ such a form has splitting type $(1^21^2)$ when
$\alpha=\beta=0$. The relative density of such $(A,B)$ is $1/p^2$, and
we therefore see that the density of $T_p((1^21^2),(11))$ is
$\frac{1}{2}(p-1)^2(p+1)/p^4$. We now compute the probability that an
element $(A,B)\in T_p((1^21^2),(11))$ is maximal at $p$. Let $(A,B)$
in $T_p((1^21^2),(11))$ be fixed. By a change of variables, we may
assume (for both $p$ odd and $p=2$) that the reduction of $(A,B)$
modulo $p$ is equal to
  \begin{equation}\label{eqsp1212}
    (\bar{A},\bar{B})=\left(\left[\begin{array}{ccc} 0 & 0 & 0 \\ 0 & 0 &
        1/2 \\ 0 & 1/2 & 0 \end{array} \right],
    \left[ \begin{array}{ccc} 1 & 0 & 0 \\ 0 & 0 & 0 \\ 0 &
        0 & 0 \end{array} \right] \right).
  \end{equation}
From Lemma 23 of \cite{BhargavaQuarticComposition}, it follows that
$(A,B)$ is maximal if and only if both $b_{22}$ and $b_{23}$ are not
divisible by $p^2$. Hence the relative density of
$\maximalatp((1^21^2),(11))$ in $T_p((1^21^2),(11))$ is
$(1-1/p)^2$. In conjunction with the previously computed density of
$T_p((1^21^2),(11))$, it follows that the density of
$\maximalatp((1^21^2),(11))$ is as stated in the proposition. The
computations of the densities of the
$\maximalatp(\varsigma_p,\varsigma_p')$ for other splitting types
$(\varsigma_p,\varsigma_p')$ are very similar to the above two
computations and so we omit them.
\end{proof}

Proposition \ref{propdenmax} computes the $p$-adic splitting densities in $V$.
However, it is possible to extract densities that are more refined. From the
splitting types with central inertia in Table~1, it is evident that
the splitting type of a $D_4$-field $L$ at a prime $p$ is not enough to
determine the decomposition groups of $L$ at $p$. However, if $v\in V(\Z)$
corresponds to $\OO_L$, then the $G(\Z_p)$-orbit of $v$ in $V(\Z_p)$ determines
$\OO_L\otimes\Z_p$, and hence determines the decomposition groups of $L$ at
$p$. In the following lemma, we compute some of these more refined densities.

If $p$ is an odd prime with splitting type $((1^21^2),(11))$ (resp.~$((2^2),(2))$) in a $D_4$-field $L$ with quadratic subfield $K$, then there are two possibilities for the
splitting type $(\varsigma_p(\phi(L)), \varsigma_p(\phi(K)))$ of $\phi(L)$ at $p$ (see Definition \ref{defphi}), namely $((1111),(11))$ (resp.~$((22),(11))$) or $((22),(2))$. Let $M_p^{(11)}(\varsigma_p,\varsigma_p')$ (resp.~$M_p^{(2)}(\varsigma_p,\varsigma_p')$) denote
the subset of elements $(A,B)\in M_p(\varsigma_p,\varsigma_p')$ corresponding to $(Q,T)$ under Theorem \ref{thm:d4parameterization} such that $\phi(\Frac(T))$ has splitting type $(11)$ (resp.~$(2)$) at $p$.  

\begin{lemma}\label{lemdecden}
Let $p$ be an odd prime and let $(\varsigma_p, \varsigma_p')$ be one of $((1^21^2), (11))$ or $((2^2), (2))$. Then with the notation of the previous paragraph,
\[ \mu(\maximalatp^{(2)}(\varsigma_p, \varsigma_p')) = \mu(\maximalatp^{(11)}(\varsigma_p, \varsigma_p')). \]
\end{lemma}
\begin{proof}
We prove the lemma only for the splitting type $((1^21^2),(11))$,
since the proof is very similar for $((2^2),(2))$. Let $(A,B)\in
V(\Z_p)$ be a maximal element, corresponding to a pair degree 4 \'etale $\bZ_p$-algebra $Q$ and a quadratic subalgebra $T$, and let $L$ (respectively, $K$) denote the fraction field of $Q$ (respectively, $T$). Assume that the splitting type of $(A,B)$ is
$((1^21^2),(11))$, and let $f(x,y)=4\det(Ax+By)$ denote the cubic
resolvent polynomial of $(A,B)$. Then the $x^3$-coefficient of $f$ is
$0$, and dividing $f$ by $y$ yields a binary quadratic
form. Propositions \ref{prop:discrelation} and \ref{propparaminv}
imply that we have
\begin{equation*}
\Disc_p(f(x,y)/y)/p^2=\Disc_p(\phi(K)),
\end{equation*}
where $\Disc_p$ means the $p$-part of the discriminant and hence, sign issues don't arise.
Since the two possible decomposition groups for the splitting type
$((1^21^2),(11))$ are determined by the splitting behaviour of $p$ at
$\phi(K)$ (see Table 1), it follows that the relative
densities of these decomposition groups can be computed by computing
the relative densities of the different possible splitting behaviours
of $p$ in the quadratic order whose discriminant is
$\Disc_p(f(x,y)/y)/p^2$.

Since $p$ is odd, from the discussion surrounding \eqref{eqsp1212}, it
follows that we may assume $(A,B)$ satisfies
\begin{equation*}
a_{11}=a_{12}=a_{13}=b_{12}=b_{13}=0,\;\;b_{11}=1,\;\;a_{23}\equiv 1\!\!\pmod{p},\;\;a_{22}\equiv b_{22}\equiv b_{23}\equiv b_{33}\equiv 0\!\!\pmod{p}.
\end{equation*}
Consider the pair $(A,B_1)$, where the $B_1$ is obtained from $B$ by
dividing the $b_{22}$, $b_{23}$, and $b_{33}$ by $p$. Let $f_1$ denote
the cubic resolvent form of $(A,B_1)$. Then the discriminant of
$f_1(x,y)/y$ is exactly the same as the discriminant of
$\phi(K)$. Hence the decomposition group of $L$ is determined by the
splitting of $p$ in $f_1(x,y)/y$. Since $(A,B)$ was assumed to be
maximal, it follows that $p$ does not divide the discriminant of
$f_1(x,y)/y$. It is easy to check that the density of elements
$(A,B_1)$ such that $f_1(x,y)/y$ has splitting type $(11)$
(resp.\ $(2)$) is exactly $1/2$, yielding the lemma.
\end{proof}

Let $\Maximalatp$ denote the set of elements $(A,B)\in V(\Z_p)$ that are
maximal at $p$, and let $\unramifiedatp$ denote the set of elements $(A,B)$ in $\Maximalatp$ that do not have central inertia, i.e., the splitting type of any $(A,B) \in \unramifiedatp$ at $p$ is not equal to  $((1^4),(1^2))$, $((1^21^2),(11))$, or $((2^2),(2))$. Summing the values obtained in Proposition \ref{propdenmax} we can compute the density of $\Maximalatp$. To determine the density of $\unramifiedatp$, we add up the values of the first 9 rows. 
\begin{thm}\label{thdenfull}
  We have
  \begin{eqnarray*}
    \mu(\Maximalatp)&=&
    \Bigl(1-\frac1{p^2}\Bigr)
    \Bigl(1-\frac1{p^2}-\frac2{p^3}+\frac2{p^4}\Bigr); \\
    \mu(\unramifiedatp) &=& \Bigl(1-\frac{1}{p^2}\Bigr)\Bigl(1 - \frac{1}{p}\Bigr)^2\Bigl(1 + \frac{2}{p}\Bigr).
  \end{eqnarray*}
\end{thm}

\section{Counting $D_4$-fields using geometry-of-numbers methods}\label{sec:goncount}

In the previous section, we defined an injective map from $D_4$-fields
to $G(\Z)$-orbits on $V(\Z)$ and determined generators $\D$ and $\q$
for the ring of invariants for the action of $G$ on $V$. In this
section, our goal is to count generic $G(\Z)$-orbits on $V(\Z)$ having
bounded invariants.

Recall that an element $v\in V(\Z)$ is said to be generic if $v$
corresponds to an order in a $D_4$-field, and the subset of elements
in $V(\Z)$ with infinite splitting type $\varsigma$ is denoted by by
$V(\Z)^{(\varsigma)}$.  For any $G(\Z)$-invariant set $\LL\subset
V(\Z)$ and for any $\delta > 0$, let $ \NNumQ^{(\delta)}(\LL;X,Y)$
denote the number of generic $G(\Z)$-orbits $v$ on $L$ such that
$X<|\q(v)|\leq (1+\delta)X$ and $Y<|\D(v)|\leq (1+\delta)Y$. Our goal
in this section is to prove the following theorem:
\begin{thm}\label{thmaincount}
  Let $X$ and $Y$ be positive real numbers going to infinity such that
  $(Y\log Y)^2=o(X)$. Then we have
  \begin{equation*}
    \NNumQ^{(\delta)}({V(\Z)^{(\varsigma)}};X,Y)=\frac{\zeta(2)}{2\tau_\varsigma}\delta^2XY+
    o_\delta(XY),
  \end{equation*}
  where $\tau_\varsigma = 8$ when $\varsigma=((1111),(11))$ or
  $\varsigma=((22),(11))$ and $\tau_\varsigma = 4$ otherwise.
\end{thm}

To do so, we study the fundamental domain for the action of the
non-reductive group $G(\bZ)$ on $V(\bR)$. We then compute the volume
of a cover of this fundamental set after cutting of the cusps in terms
of an Euler product of local densities.

\subsection{Construction of fundamental domains}
In this section, our goal is to construct a finite cover for a
fundamental domain for the action of $G(\Z)$ on $V(\R)$. As a first
step, we describe the $G(\R)$-orbits on $V(\R)$, and the sizes of the
stabilizers in $G(\R)$ of elements in each orbit. Before we do so, it
will be convenient to introduce the following group and space: Let
$V_\red \subset V$ consist of all pairs $(A,B)$ of the form
\begin{equation}\label{eqABV}
(A,B)=\left(
\begin{bmatrix}
0&0&0\\
0&a_{22}&\frac{a_{23}}{2}\\
0&\frac{a_{23}}{2}&a_{33}
\end{bmatrix},
\begin{bmatrix}
1&0&0\\
0&b_{22}&\frac{b_{23}}{2}\\
0&\frac{b_{23}}{2}&b_{33}
\end{bmatrix}
\right).
\end{equation}
The subgroup $G_\red$ of $G$ acts on
$V_\red$, where $G_\red$ consists of elements $(g_2,g_3) \in \GL_2(\bZ) \times
\SL_3(\bZ)$ such that 
\begin{equation}\label{eqG}
g_2 = \begin{bmatrix} \pm1 & 0 \\ \ast & 1 \end{bmatrix}, \hspace{1cm}
\mathrm{and} \hspace{1cm} g_3 = \begin{bmatrix} \pm 1 & 0 & 0 \\ 0 &
  \ast & \ast \\ 0 & \ast & \ast
\end{bmatrix}.
\end{equation}
where the lower $2\times 2$ submatrix of $g_3$ is an element of
$\SL_2^\pm$. We have the following result.

\begin{prop}\label{proporbstabR}
The orbits for the action of $G(\R)$ on the set of elements in $V(\R)$
having nonzero invariants $\q$ and $\D$ corresond to a pair of \'etale algebras
$(L_\infty, K_\infty)$ with splitting types and invariants as follows:
\begin{itemize}
\item[{\rm (1)}] When $\q > 0$ and $\D > 0$, there are two orbits, one with
  splitting type $((1111),(11))$ and one with splitting type $((22),(11))$;
\item[{\rm (2)}] When $\q < 0$ and $\D > 0$, there is one orbit with splitting type $((112),(11))$;
\item[{\rm (3)}] When $\q > 0$ and $\D < 0$, there is one orbit has splitting type $((22),(2))$.
\end{itemize}
The respective sizes of the stabilizers in $G(\R)$ of elements in
these orbits are $8$ in the first case, and $4$ in the second and
third cases.
\end{prop}
\begin{proof}
We start with a few observations. First note that $G(\R)$-orbits on
$V(\R)$ having fixed invariants $\D$ and $\q$ are in bijection
with $G_\red(\R)$-orbits on $V_\red(\R)$ having invariants $\q$
and $\D$. This is because $(A,B)\in V(\R)$ is clearly
$G(\R)$-equivalent to some $(A_\red,B_\red)\in
V_\red(\R)$. Futhermore, if two elements in $V_\red(\R)$ are
$G(\R)$-equivalent via some $g\in G(\R)$, then $g$ must in fact belong
to $G_\red(\R)$. This latter fact also implies that the stabilizer in
$G(\R)$ of any element in $V_\red(\R)$ is contained in
$G_\red(\R)$. Also note that $G_\red(\R)$-orbits on $V_\red(\R)$
having nonzero invariants $\D$ and $\q$ are in bijection with
$G_\red(\R)$-orbits on $V_\red(\R)$ having invariants $\q/|\q|$
and $\D/|\D|$. Indeed, if $(A,B)\in V_\red(\R)$ has invariants
$\q$ and $\D$, then dividing $A$ by $\sqrt{|\D|}$ and
dividing the lower $2\times 2$-submatrix of $B$ by
$\sqrt{|\q|/|\D|}$ yields the necessary bijection. Moreover, the
stabilizers in $G_\red(\R)$ of these two elements are clearly the
same. Therefore, it suffices to prove the proposition in the case when
$\q$ and $\D$ are $\pm1$.

Consider the case $\q=\D=1$. Let $(A,B)\in V_\red(\R)$ have such
invariants. By replacing $(A,B)$ with a $G_\red(\R)$-translate, we
transform it as follows: first, we ensure that $a_{22}=a_{33}=0$;
next, we subtract an appropriate multiple of $A$ from $B$ to ensure
that its off-diagonal entries are $0$; finally, we use an element of
$\SL_2(\R)\subset G_\red(\R)$ to ensure that $|b_{22}|=|b_{33}|$. From
the fact that $\q=\D=1$, it follows that we have transformed
$(A,B)$ into the form
  \begin{equation*}
    \left(\begin{array}{ccc}{0}&{}&{}\\{}&{}&\frac12\\{}&\frac12&{}\end{array}\right),
    \left(\begin{array}{ccc}1&{}&{}\\{}&{\frac{\pm1}{4}}&{}\\{}&{}&\frac{\pm1}{4}
    \end{array}\right),
  \end{equation*}
where $b_{22}$ and $b_{33}$ are either both positive or both negative.
It is easy to check that $(A,B)$ has splitting type $((22),(11))$ in
the former case and $((1111),(11))$ in the latter case. Furthermore,
the stabilizer in $G_\red(\R)$ of $(A,B)$ in either case is seen to
consist of the following eight elements.
\begin{equation*}
  \begin{array}{ll}&\scriptstyle{
\left(\begin{bmatrix} 1 &  \\  & 1 \end{bmatrix}, \begin{bmatrix} 1 &  &  \\  &
  1 &  \\  &  & 1
\end{bmatrix}\right),
\left(\begin{bmatrix} -1 &  \\  & 1 \end{bmatrix}, \begin{bmatrix} -1 &  &  \\  &
  1 &  \\  &  & -1
\end{bmatrix}\right),
\left(\begin{bmatrix} 1 &  \\  & 1 \end{bmatrix}, \begin{bmatrix} -1 &  &  \\  &
   & 1 \\  & 1 & 
\end{bmatrix}\right),
\left(\begin{bmatrix} -1 &  \\  & 1 \end{bmatrix}, \begin{bmatrix} 1 &  &  \\  &
   & 1 \\  & -1 & 
\end{bmatrix}\right),}\\[.3in]
    
      &\scriptstyle{
\left(\begin{bmatrix} 1 &  \\  & 1 \end{bmatrix}, \begin{bmatrix} 1 &  &  \\  &
  \!\!\!\!-1 &  \\  &  & \!\!\!\!-1
\end{bmatrix}\right),
\left(\begin{bmatrix} -1 &  \\  & 1 \end{bmatrix}, \begin{bmatrix} -1 &  &  \\  &
  -1 &  \\  &  & 1
\end{bmatrix}\right),
\left(\begin{bmatrix} 1 &  \\  & 1 \end{bmatrix}, \begin{bmatrix} -1 &  &  \\  &
   & \!\!\!\!-1 \\  & \!\!\!\!-1 & 
\end{bmatrix}\right),
\left(\begin{bmatrix} -1 &  \\  & 1 \end{bmatrix}, \begin{bmatrix} 1 &  &  \\  &
   & -1 \\  & 1 & 
\end{bmatrix}\right)}.
 \end{array}
\end{equation*}
This concludes the proof of the first item in Proposition
\ref{proporbstabR}. We omit the proofs of the other two items since
they are very similar.
\end{proof}

Recall that the set of elements in $V(\R)$ with infinite splitting
type $\varsigma$ is denoted by $V(\R)^{(\varsigma)}$. Given a
splitting type $\varsigma$, let $(A_\varsigma,B_\varsigma)\in
V(\R)^{(\varsigma)}\cap V_\red(\R)$ be an element whose invariants
have absolute value $1$. By multiplying $A_\varsigma$ by $\sqrt{|\D|}$
and multiplying the bottom $2\times 2$ submatrix of $B$ by
$\sqrt{|\q|/|\D|}$, we obtain an element with invariants $\q$ and
$\D$, for any pair $(\q,\D)\in\R^2$ having the appropriate signs. We
thus obtain the following result which follows immediately from
Proposition~\ref{proporbstabR}.

\begin{prop}
Fix an infinite splitting type $\varsigma$. There exists a fundamental
set $\fundset$ for the action of $G(\R)$ on $V(\R)^{(\varsigma)}$ such
that $\fundset$ contains one element $(A,B)$ having invariants $\q$
and $\D$ for any $(\q,\D)\in\R^2$ having the appropriate
signs. Moreover, $\fundset$ may be constructed so that the element
$(A,B)\in \fundset$ having invariants $\q$ and $\D$ is such that the
coefficients of $A$ are bounded by $O_\delta(|\D|^{1/2})$ and the
coefficients of $B$ are bounded by $O_\delta(|\q|^{1/2}|\D|^{-1/2})$.
\end{prop}

Let $\FF$ be a fundamental domain
for the action of $G(\Z)$ on $G(\R)$. We may assume that $\FF$ is
contained in the Siegel domain $\Ss=\Ss_1\Ss_2$, where
\begin{equation}\label{eqsiegel}
  \begin{array}{rcl}
\Ss_1\!\!\!&=&\!\!\!\left\{\left(\begin{bmatrix}1 & \\ n & 1\end{bmatrix},
\begin{bmatrix} 1 &&\\&1&\\&m_3&1\end{bmatrix}
\begin{bmatrix} 1 &&\\&\!\!t^{-1}&\\&&\!\!\!t\end{bmatrix}
\begin{bmatrix} 1 &&\\&\cos\theta&\sin\theta\\&-\sin\theta&\cos\theta
\end{bmatrix}
\right):n,m_3\in[0,1),\;t>\frac12\right\},\\[.3in]
\Ss_2\!\!\!&=&\!\!\!
\left\{\left(\begin{bmatrix}1 & \\  & 1\end{bmatrix},
  \begin{bmatrix} 1 &&\\m_1&1&\\m_2&&1\end{bmatrix}\right):m_1,m_2\in[0,1)
      \right\}.
  \end{array}
\end{equation}
We have $\FF=\FF_2\FF_1$, where $\FF_1\subset\Ss_1$ and $\FF_2=\Ss_2$.
From an argument identical to that in \cite[\S2.1]{BSBQ}, it follows that
$\FF\cdot \fundset$ is a cover of a fundamental domain for the action
of $G(\Z)$ on $V(\R)^{(\varsigma)}$, where the $G(\Z)$-orbit of $v$ is
represented $m(v)$ times. Here $m(v)$ is given by
\begin{equation*}
m(v)=\#\Stab_{G(\R)}(v)/\#\Stab_{G(\Z)}(v).
\end{equation*}
Every element in $V(\R)$ is fixed by the element $({\rm Id},g_3)\in G(\Z)$,
where $g_3$ is the diagonal $3\times 3$ matrix whose diagonal entries
are $1$, $-1$, and $-1$. Since the set of elements in $V(\R)$, that
have a stabilizer in $G(\Z)$ of size greater than $2$, has measure $0$,
we obtain the following theorem.

\begin{thm}\label{thfunddom}
The multiset $\FF\cdot \fundset$ is an $(\tau_\varsigma)/2$-fold cover of
a fundamental domain for the action of $G(\Z)$ on $V(\R)^{(\varsigma)}$,
where $\tau_\varsigma=8$ for $\varsigma=((1111),(11))$ or $((22),(11))$ and
$\tau_\varsigma=4$ for $\varsigma=((112),(11))$ or $((22),(2))$.
\end{thm}

\subsection{Averaging and cutting off the cusp}

Let $\LL\subset V(\R)^{(\varsigma)}$ be a $G(\Z)$-invariant
lattice, and denote the set of generic elements in $\LL$ by
$\LL^\gen$. Given a subset $W$ of $V(\R)$ and a constant $\delta > 0$,
we denote the set of elements $w\in W$ with $X\leq|\q(w)|<(1+\delta)X$
and $Y\leq|\D(w)|<(1+\delta)Y$ by $W_{XY}$.  Since the stabilizer in
$G(\Z)$ of a generic element in $V(\Z)$ has size $2$, Theorem
\ref{thfunddom} implies that we have
\begin{equation}\label{eqavg1}
  \NNumQ^{(\delta)}(\LL;X,Y)=\frac{2}{\tau_\varsigma}
  \#\{\FF\cdot \fundset_{XY}\cap \LL^\gen\}.
\end{equation}

We now pick the following bounded open nonempty subset $G_0$ of
$G_\red(\R)$:
\begin{equation*}
  G_0:=\left\{\left(\begin{bmatrix}1 & \\ n & 1\end{bmatrix},
    \begin{bmatrix}1&&\\&a&b\\&c&d
    \end{bmatrix}\right):n\in\scriptstyle{(0,\frac{\sqrt{Y}}{X})},\;
\left(\begin{smallmatrix}a&b\\c&d
\end{smallmatrix}\right)\in G_1\subset\SL_2(\R)
    \right\},
\end{equation*}
where $G_1$ is a bounded open nonempty $\SO_2(\R)$-invariant subset of
$\SL_2(\R)$. The reason for the choice of the range of $n$ will become
apparent in what follows.  Replacing the fundamental domain $\FF$ in
\eqref{eqavg1} by $\FF_2\FF_1 g$ and averaging over $g\in G_0$, and
following the argument in the proof of \cite[Theorem 2.5]{BSBQ}, we
    obtain
\begin{equation*}
  \begin{array}{rcl}
  \NNumQ^{(\delta)}(\LL;X,Y)&=&\displaystyle\frac{2}{\tau_\varsigma\Vol(G_0)}\int_{g\in G_0}
  \#\{\FF_2\FF_1 g\cdot \fundset_{XY}\cap \LL^\gen\}dg\\[.2in]
&=&\displaystyle\frac{2XY^{-1/2}}{\tau_\varsigma\Vol(G_1)}\int_{g\in\FF_1}
  \#\{(\FF_2gG_0\cdot \fundset_{XY}\cap \LL^\gen\}dg,
  \end{array}
\end{equation*}
since $\Vol(G_0)=X^{1/2}Y^{-1}\Vol(G_1)$.

Using coordinates from \eqref{eqsiegel}, we write an element in
$\FF_1$ as $(n,m_3,t,\theta)$. In the next lemma we show that
$\FF_2gG_0\cdot \fundset_{XY}$ has no integral generic points if
$t$ is too large.
\begin{lemma}
Suppose $\FF_2gG_0\cdot \fundset_{XY}\cap V(\Z)^\gen$ is
nonempty for $g=(n,m_3,t,\theta)$. Then $t\ll Y^{1/4}$.
\end{lemma}
\begin{proof}
  Let $(A,B)$ be an element of $\fundset$. Then we have
  $|a_{22}|\ll Y^{1/2}$. Therefore, there exists a constant $C$ such
  that if $t>CY^{1/4}$, then $|a_{22}|<1$ for every $(A,B)\in t\theta
  G_0\fundset$. The action of $m_3$, $n$, and $\FF_2$ does not
  change the value of $a_{22}$, and it follows that we have $a_{22}=0$
  for every $(A,B)\in \FF_2gG_0\cdot \fundset_{XY}\cap V(\Z)$. The
  lemma now follows since any such $(A,B)$ is not generic.
\end{proof}

We let $\FF'=\FF_2\FF_1'\subset \FF$ consist of all elements $g_2g_1$
with $g_2\in\FF_2$ and $g_1=(n,m_3,t,\theta)$, where $t\leq CY^{1/4}$
for the $C$ in the proof of the above lemma. For any lattice $\LL$ of
$V(\Z)$, define
\begin{equation}\label{eqpostavgdav}
  \NNumQ^*(\LL;X,Y) :=\displaystyle\frac{2X^{-1/2}Y}{\tau_\varsigma\Vol(G_1)}\int_{g\in\FF_1'}
  \#\{(\FF_2gG_0\cdot \fundset_{XY}\cap \LL\}dg.
\end{equation}
We use the following result of Davenport \cite{DavLemma} to estimate
$\NNumQ^*(\LL;X,Y)$:
\begin{prop}\label{davlem}
  Let $\mathcal R$ be a bounded, semi-algebraic multiset in $\R^n$
  having maximum multiplicity $m$, and that is defined by at most $k$
  polynomial inequalities each having degree at most $\ell$.    
  Then the number of integral lattice points $($counted with
  multiplicity$)$ contained in the region $\mathcal R$ is
\[\Vol(\mathcal R)+ O(\max\{\Vol(\bar{\mathcal R}),1\}),\]
where $\Vol(\bar{\mathcal R})$ denotes the greatest $d$-dimensional 
volume of any projection of $\mathcal R$ onto a coordinate subspace
obtained by equating $n-d$ coordinates to zero, where 
$d$ takes all values from
$1$ to $n-1$.  The implied constant in the second summand depends
only on $n$, $m$, $k$, and $\ell$.
\end{prop}
\noindent In fact, the proof of the above proposition implies that we
may replace $\Vol(\bar{\mathcal R})$ by the maximum of the
$d$-dimensional volumes of the projections of any unipotent translate
of $\mathcal{R}$.

Now, for $g\in\FF_1$, the set $gG_0\cdot R_{XY}^{(\varsigma)}$ is a
bounded set contained in $V_\red(\R)$. Hence, the $b_{12}$- and
$b_{13}$-coefficients of elements in $\FF_2gG_0\cdot
R_{XY}^{(\varsigma)}$ must lie in $[0,2)$. They can only be integral
when they are $0$ or $1$. Therefore, every integral point in
$\FF_2gG_0\cdot R_{XY}^{(\varsigma)}$ lies on one of four hyperplanes
in $V(\R)$: the hyperplanes corresponding to
$(b_{12},b_{13})=(0,0)$, $(0,1)$, $(1,0)$, and $(1,1)$. Moreover,
these hyperplanes are unipotent translates of each other, in fact,
by the elements in $\FF_2$ with $m_1,m_2\in\{0,1/2\}$. It follows
that the four hyperplane sections have the same volume, and Proposition
\ref{davlem} applied to them yields the same error estimates.
Therefore, we have
\begin{equation}\label{eqavg2}
 \NNumQ^*(\LL;X,Y)=\displaystyle\frac{8X^{-1/2}Y}{\tau_\varsigma\Vol(G_1)}\int_{g\in\FF_1'}
  \Vol_\LL(gG_0\cdot \fundset_{XY})dg+O({\mathcal E}(X,Y)),
\end{equation}
where $\Vol_\LL$ is computed with Euclidean measure normalized so that
$\LL$ has covolume $1$, and
\begin{equation*}
  {\mathcal E}(X,Y)=X^{-1/2}Y
  \int_{g=(0,0,t,0)\in\FF_1'}{\rm MP}(gG_0\cdot \fundset_{XY})t^{-2}d^\times t.
\end{equation*}
The quantity ${\rm MP}(g)$ denotes the maximal volume of the
projections of $gG_0\cdot \fundset_{XY}$ onto its
coordinate-hyperplanes. Every element $(A,B)$ in $\fundset$ is
such that the coefficients of $A$ are bounded by $Y^{1/2}$ and the
coefficients of $B$ are bounded by $X^{1/2}/Y^{1/2}$. By construction
of $G_0$, the same is true for every element in $G_0\cdot
\fundset$.  Then the error integral is easily bounded: as long as
{$Y\ll X$}, the maximum projection is on to the coordinate subspace
obtained by setting $b_{22}$ to $0$ since the ranges of the other
coordinates are clearly $\gg 1$ for every value of $g\in\FF_1'$. We
therefore have
\begin{equation}\label{eqeb}
{\mathcal E}(X,Y)\ll YX^{-1/2}\int_{t=1}^{Y^{1/4}}Y^{1/2}Cd^\times
t\ll Y^{3/2}X^{1/2}\log Y.
\end{equation}

We next have the following bound on the number of non-generic
$G(\Z)$-orbits on $V(\Z)$.
\begin{prop}\label{propredgon}
  We have
  \begin{equation*}
\frac{X^{-1/2}Y}{\Vol(G_1)}\int_{g\in \FF_1'} \#\{gG_0\cdot
\fundset_{XY}\cap V(\Z)\backslash V(\Z)^\gen\}dg =o(XY).
  \end{equation*}
\end{prop}
\begin{proof}
If $v\in V(\Z)$ is not generic, then there exists unramified splitting
types $\varsigma'' = (\varsigma''_p,\varsigma'''_p)_p$ for all primes
$p$ such that $(\varsigma_p(v),\varsigma_p'(v))\neq(\varsigma''_p,\varsigma'''_p)$ for
all primes $p$. Given any unramified splitting type, there exists a
constant $c(\varsigma''_p,\varsigma'''_p)<1$ such that the density of
elements in $V(\Z_p)$ that do not have splitting type
$(\varsigma''_p,\varsigma'''_p)$ is bounded above by
$c(\varsigma''_p,\varsigma'''_p)$. From \eqref{eqfinalcountgon}, we
therefore obtain for any fixed integer $M$:
  \begin{equation*}
\frac{X^{-1/2}Y}{\Vol(G_1)}\int_{g\in \FF_1'} \#\{gG_0\cdot
\fundset_{XY}\cap V(\Z)\backslash V(\Z)^\gen\}dg
    \ll \sum_{\varsigma''}XY\cdot\prod_{p<M}c(\varsigma''_p,\varsigma'''_p).
  \end{equation*}
  Letting $M$ tend to infinity, we obtain the result.
\end{proof}

From \eqref{eqavg2}, \eqref{eqeb} and Proposition \ref{propredgon}, we
see that if $X$ and $Y$ go to infinity such that $Y(\log Y)^2=o(X)$, then
\begin{equation}\label{eqfinalcountgon}
  \begin{array}{rcl}
    \NNumQ^{(\delta)}(\LL;X,Y)&=&\displaystyle\frac{8X^{-1/2}Y}{\tau_\varsigma\Vol(G_1)}
    \Vol(\FF_1)\Vol_\LL(G_0\cdot
  \fundset_{XY})+o(XY)\\[.2in]
  &=&\displaystyle\frac{8}{\tau_\varsigma\Vol(G_0)}\Vol(\FF_1)\Vol_\LL(G_0\cdot
  \fundset_{XY})+o(XY).
  \end{array}
\end{equation}

To compute the volumes of $G_0\cdot \fundset_{XY}$, we have the
following result, which follows immediately from a Jacobian change of
variables computation.
\begin{prop}\label{thjac}
  Let $dv_i$ be the standard Euclidean measures on $V_\red(\R)$, let
  $dh$ denote the Haar-measure on $G_\red(\R)$ obtained from the
  $\bar{N}AN$ decomposition of $\SL_2(\R)$, and let $d\D$ $d\q$ be a
  measure on $\fundset$. We have a natural map $G_\red(\R)\times
  \fundset\to V_\red(\R)$. Then the Jacobian change of variables is
  $1/16$, i.e., for any measurable function $\varphi$ on $V_\red(\R)$,
  we have
\begin{equation}
  \int_{v\in G_\red(\R)\cdot \fundset}\varphi(v)dv=
  \frac1{16}\int_{r\in\fundset}
  \int_{h\in G_\red(\R)}\varphi(g\cdot r)dh\,d\D(r)\, d\q(r).
\end{equation}
\end{prop}

Therefore, we obtain the following theorem from which Theorem \ref{thmaincount}
follows immediately.
\begin{thm}\label{thgengoncount}
Let $\LL$ denote a finite union of $G(\Z)$-invariant lattice in
$V(\R)^{(\varsigma)}$. Then, for positive real numbers $X,Y$ going to
infinity such that $Y(\log Y)^2=o(X)$, we have
\begin{equation*}
  \NNumQ^{(\delta)}(\LL;X,Y)=
  \frac{\zeta(2)}{2\tau_\varsigma}\delta^2XY\prod_{p}\Vol(\LL_p)+o_\delta(XY),
\end{equation*}
where $\LL_p$ denotes the closure of $\LL$ in $V(\Z_p)$, the volumes
of sets in $V(\Z_p)$ are taken with respect to the usual Euclidean
measure, and $\tau_\varsigma$ are as in Theorem \ref{thmaincount}.
\end{thm}
\begin{proof}
The theorem follows from \eqref{eqfinalcountgon} and Proposition
\ref{thjac} since the volume of $\FF_1$ under the measure $dh$ is
$\zeta(2)$, the volume of $\fundset_{XY}$ is $XY$ and $\Vol_\LL$
differs from normal Euclidean measure by a factor of
$\prod_{p}\Vol(\LL_p)$.
\end{proof}

\section{Uniformity estimates and sieving to maximal $D_4$-orders}

In order to use our results from \S4, \S5, and \S6, \S7 to prove our main
theorems, we will employ simple sieves. In this section, we start by
collecting the requisite tail estimates. First, we need a bound on the number of $D_4$-fields having
central inertia at some large prime, which is established in \S8.1. On the other hand, we obtain an estimate on the
number of $G(\Z)$-orbits on $V(\Z)$ that are non-maximal at some large
prime  in \S8.2. It is interesting to note that the results
in \S7 are not strong enough for these estimates, and so we
employ techniques from \S4.

\subsection{Bounding the number of $D_4$-fields with large central inertia}

We start with a preliminary lemma bounding the number of $D_4$-fields with fixed conductor. 

\begin{lem}\label{keyuniflem}
For any positive integer $N$, the number of $D_4$-fields with
conductor $N$ is bounded by $O_\epsilon(N^\epsilon)$.
\end{lem}
\begin{proof}
Let $L$ be a $D_4$-field with conductor $N$, and let $K$ be the
quadratic subfield of $L$. Then the discriminant of $K$ divides
$N$. Hence the number of choices for $K$ is bounded by twice the
number of divisors of $N$. Given a fixed quadratic field $K$ whose
discriminant $D$ divides $N$, the number of quartic $D_4$-fields of
conductor $N$ whose quadratic subfield is $K$ is bounded by
$4 \cdot \#\Cl_2(K)$ times the number of squarefree ideals dividing
$4N$ (see \S3 of \cite{CDOQuartic}). But $4 \cdot \#\Cl_2(K) \ll_\epsilon
D^\epsilon$ and the number of divisors of $4N$ is $\ll_\eps N^\eps$.
Combining these estimates yields the lemma.
\end{proof}

Next, we prove the required estimate on $D_4$-fields having central inertia at large primes $p$ by combining the previous lemma with Lemma \ref{lemunifcondp}.
\begin{prop}\label{propunifJ}
Fix $\epsilon > 0$. Let $X$ and $Y$ be integers such that $X\geq
Y$. Then the number of $D_4$-fields $L$ such that $X\leq
\q(L)<2X$, $Y\leq \D(L)<2Y$, and $L$ has central inertia
at a prime $p$ is bounded by $O_\epsilon(XY/p^{2-\epsilon})$.
\end{prop}
\begin{proof}
We consider two ranges of $p$. We fix a large positive real number $M$
(any $M>16$ will suffice). When $p\geq X^{1/M}$, we have $p\geq
(XY)^{1/2M}$. The number of possible conductors for a $D_4$-field $L$
satisfying the conditions of the proposition is bounded by
$O(XY/p^2)$, since the conductor is bounded by $4XY$ and is divisible
by $p^2$. Hence, from Lemma \ref{keyuniflem}, it follows that the
number of such fields $L$ is bounded by
$O_\epsilon((XY)^{1+\epsilon}/n)=O_\epsilon(XY/n^{1-\epsilon})$.

For $p\leq X^{1/M}$, we see from Lemma \ref{lemunifcondp} that the
number of $D_4$-fields satisfying the conditions of the proposition
and having splitting type $((1^21^2),(11))$ or $((2^2),(2))$ is
bounded by
\begin{equation*}
  O\Bigl(\frac{X}{p^{2-\epsilon}}\cdot\sum_{\substack{[K:\Q]=2\\|\Disc(K)|<Y}}L(1,K/\bQ)\Bigr)
\end{equation*}
which is bounded by $O(XY/n^{1-\epsilon})$ by \S3 of \cite{Siegel}.
Similarly, the number of $D_4$-fields satisfying the conditions of the
proposition and having splitting type $((1^4),(1^2))$ is bounded by
\begin{equation*}
  O\Bigl(\frac{X}{p^{1-\epsilon}}\cdot\sum_{\substack{[K:\Q]=2\\p\mid\Disc(K)\\|\Disc(K)|<Y}}L(1,K/\bQ)\Bigr).
\end{equation*}
When $Y$ grows faster than a large power of $p$, say $Y\gg p^M$, the
sum is bounded by $O(Y/p)$ by arguements identical to those in
\S3 of \cite{Siegel}. When $p\gg Y^{1/M}$, we instead use the classical
bound of $O_\epsilon(\Disc(K)^\epsilon)$ on $L(1,K/\bQ)$ to obtain the
proposition.
\end{proof}

\begin{cor}\label{corunif}
Let $X$ and $Y$ be integers such that $X\geq Y$. Then the number of
$D_4$-fields $L$ such that $X\leq \q(L)<2X$, $Y\leq \D(L)<2Y$, and
$J(L)=n$ is bounded by $O_\epsilon(XY/n^{1-\epsilon})$.
\end{cor}
\begin{proof}
 If $p\mid J(L)$ for some odd prime $p$, then by Proposition
 \ref{prop:discrelation}, the $p$-part of $J(L)$ is $p^2$ and the
 splitting type of $L$ at $p$ is $((1^21^2),(11))$ or
 $((2^2),(22))$. The corollary therefore follows from a proof
 identical to that of Proposition \ref{propunifJ}.
\end{proof}

\subsection{Bounding the number of non-maximal $G(\Z)$-orbits on $V(\Z)$}

For a fixed prime $p$, let $\W_p$ denote the set of generic elements
in $V(\Z)$ that correspond to nonmaximal orders in $D_4$-fields. Our
next goal is to prove a uniform tail estimate for the number of
$G(\Z)$-orbits on $\W_p$ having bounded invariants. We start with the
following lemmas.

\begin{lem}\label{lembounddyadic1}
The number of $D_4$-fields $L$ with $|\q(L)|<X$ and $|\D(L)|<Y$ is
bounded by $O(XY)$.
\end{lem}
\begin{proof}
From Lemma \ref{lemquadcount},
we see that for $X\geq Y$ (in fact for $Y\ll_\epsilon
X^{3-\epsilon}$), the number of $D_4$-fields with invariants $\q$ and
$\D$ less than $X$ and $Y$, respectively, is bounded by
$$
O\Bigl(X\cdot\sum_{|\Disc(K)|<Y}L(1,K/\bQ)\Bigr),
$$
which is $O(XY)$ by the results in \S3 of \cite{Siegel}.

When $X < Y$, we bound the number of $D_4$-fields $L$ by instead
bounding the number of fields $\phi(L)$ (see Definition \ref{defphi}). If $L$ satisfies the conditions of the lemma,
then it follows that $X/n^2\leq|\D(\phi(L))|<2X/n^2$,
$Yn^2\leq|\q(\phi(L))|<2Yn^2$, where $n=J(L)=J(\phi(L))$ is a square at
every odd prime factor. From Corollary \ref{corunif}, it follows that the
number of such $\phi(L)$ is bounded by
\begin{equation*}
O\Bigl(\sum_{n}XY/n^{1-\epsilon} \Bigr),
\end{equation*}
where $n$ runs over integers that are exactly divisible by $p^{2}$ for every odd
prime factor. Since the sum converges, the lemma follows.
\end{proof}

We now prove the following uniform bound on the number of
$G(\Z)$-orbits on $\W_p$, the set of generic elements in $V(\Z)$ that
are not maximal at $p$.
\begin{prop}\label{propunifmax}
The number of $G(\Z)$-orbits $v$ on $\W_p$ with $X\leq \q(v)<2X$ and
$Y\leq\D(v)<2Y$ is bounded by $O_\epsilon(XY/p^{2-\epsilon})$.
\end{prop}
\begin{proof}
An element $v\in\W_p$ gives a quartic ring $Q$ whose field of
fractions $L$ is a $D_4$-field. Let $i(v)$ denote $C(v)/C(L)$, the
ratio of the conductors of $v$ and $L$. From Lemma \ref{lemsqfull}, it
follows that $i(v)$ is divisible by $p^2$. From Lemma
\ref{lembounddyadic1}, it follows that the number of possible fields
$L$ that occur this way is bounded by $O(XY/i(v)^{1-\epsilon})$. Next, note that
the index of $Q$ in the ring of integers of $L$ divides $i(v)$. The
methods of \cite{JN} imply that the number of suborders of index
$k=\prod p_i^{e_i}$ of a maximal quartic ring is bounded by
$$j(k)\defeq\prod p_i^{(2+\epsilon)\lfloor\frac{e_i}{4}\rfloor}.$$ Once
the order $Q$ has been determined, there are $O(1)$ choices for the quadratic subring $T$ of $Q$
corresponding to $v$ under Theorem \ref{thm:d4parameterization}. Finally,
Corollary 4 of \cite{BhargavaQuarticComposition} asserts that the number
of cubic resolvents of ring $Q$ is $d(c)$, the sum of the divisors of
the content $c$ of $Q$. Furthermore, the content $c$ of the quartic
ring corresponding to $v=(A,B)$ is  equal to be the gcd of the coefficients
of $A$ (see
\S3.6 of \cite{BhargavaQuarticComposition}), which implies that $i(v)$ is a multiple of $c^4$.

Therefore, it follows that the number of $G(\Z)$-orbits on $\W_p$
satisfying the conditions of the proposition is bounded by
\begin{equation}\label{equnifbd}
\sum_{p^2\mid m}\sum_{c^4\mid m}\sum_{k\mid m}j(k)d(c)(X/m^{1-\epsilon}),
\end{equation}
where $m$ runs over all integers divisible by $p^2$.  Using the
multiplicativity of $j$ and $d$, the expression \eqref{equnifbd} is
easily seen to be $\ll_\epsilon X/p^{2-\epsilon}$ and the proposition
follows.
\end{proof}

\section{Proof of the main theorems}\label{sec:proofs}
In this section, we conclude the proof of the generalization of Theorem \ref{5.3} that allows for imposing certain local specifications (see Theorem \ref{thacceptablecount}). As a byproduct of the two asymptotics obtained for $\NumC(\Sigma;X,X^{1/2})$ by Theorem \ref{thm:analyticmain2} and in \S9.1, we prove Theorem \ref{stablefamilies}. 

Theorems \ref{MAINTHEOREM} and \ref{congruence conditions} are proved in \S9.2. First, we obtain asymptotics for $\NumC(\Sigma;X,Y)$ when $Y > X^{1/2}$ from counting $D_4$-fields $\phi(L)$ with conductor bounded by $X$ and small quadratic discriminant, where $L \in \LL(\Sigma)$.  We then prove Theorem \ref{congruence conditions} after employing a bound that follows from the analytic methods in Section 4 and 5. Finally, in \S9.3 we prove a refinement of Theorem \ref{classgroups}, which follows from Theorem \ref{congruence conditions} in conjunction with the $p$-adic volumes determined in Proposition \ref{propdenmax}.

\subsection{A refinement of Theorem \ref{5.3} and
  the proof of Theorem \ref{stablefamilies}}

Recall that for a collection of local specifications $\Sigma$, $\LL(\Sigma)$ is the set of $D_4$-fields
$L$ such that the pair consisting of the splitting type of $L$ and the splitting type of $K$, its quadratic
subfield at a prime $p$ (respectively, at $\infty$) is contained in $\Sigma_p$
for all $p$ (respectively, in $\Sigma_\infty$).  A set $\Sigma$ of
local specifications (and the corresponding family $\LL(\Sigma)$) is said to be {\it
  acceptable} if for all but finitely many primes $p$, the set
$\Sigma_p$ contains all unramified splitting types and tamely ramified splitting types without central inertia. (In the notation of \eqref{d4diagram}, $\Sigma_p$ contains exactly the pairs  $(\varsigma_p(L_1),\varsigma_p(K_1))$ contained in the first two groups in Table 1.)

Recall that for a prime $p$ and a splitting type
$(\varsigma_p,\varsigma_p')$, we computed the density
$\mu(\maximalatp(\varsigma_p,\varsigma_p'))$ in Proposition
\ref{propdenmax}.  We define the density $\mu(\Sigma_p)$ to be the sum
of the values of $\mu(\maximalatp(\varsigma_p,\varsigma_p'))$ over
$(\varsigma_p,\varsigma_p')\in\Sigma_p$, and define the density of
$\mu(\Sigma_\infty)$ to be the sum of $1/\tau_{\varsigma_{\infty}}$ over
$\varsigma_{\infty}\in\Sigma_\infty$ (see Theorem \ref{thmaincount} for the definition of $\tau_{\varsigma_{\infty}}$.)  

For positive real numbers $X$ and $Y$,
let $\NNumQ^{(\delta)}(\Sigma;X,Y)$ be as in Section~\ref{sec:goncount}. We have the
following theorem, giving another proof that the heuristics of \eqref{eqHXY} holds
for certain ranges of $X$ and $Y$.
\begin{thm}\label{thacceptablecount}
Let $\Sigma$ be an acceptable set of local specifications. For
positive real numbers $X$ and $Y$ such that $Y(\log Y)^2=o(X)$, we
have
\begin{equation*}
  \NNumQ^{(\delta)}(\Sigma;X,Y)=
  \frac{\zeta(2)}{2}\cdot\delta^2\cdot \mu(\Sigma_\infty)\cdot \prod_p\mu(\Sigma_p) \cdot XY +o_\delta(XY).
\end{equation*}
\end{thm}
\begin{proof}
By Proposition \ref{propparam}, it suffices to obtain asymptotics for
the number of generic $G(\Z)$-orbits $(A,B)$ on $V(\Z)$ such that
$(A,B)$ is maximal and the splitting type of $(A,B)$ at each place
belongs to $\Sigma$. The number of generic $G(\Z)$-orbits on $V(\Z)$
satisfying any finite set of congruence conditions has been estimated
in Theorem \ref{thgengoncount}. Theorem \ref{thacceptablecount} then
follows from Theorem \ref{thgengoncount} and the uniformity estimates
in Propositions \ref{propunifJ} and \ref{propunifmax} by means of a
simple sieve. We omit the details since they are very similar to those in the proof of Theorem \ref{thlargedcount}.
\end{proof}

As a consequence of the main terms obtained in Theorems
\ref{thm:analyticmain2} and \ref{thacceptablecount}, we may now prove
Theorem \ref{stablefamilies}. It is interesting to note that the proof of Theorem \ref{stablefamilies} is much more involved than that of Theorem \ref{2}. In particular, it is not clear that the arguments in Section 5 can be refined to directly allow for imposing acceptable sets of local specifications.

\vspace{.1in}
\noindent{\it Proof of Theorem \ref{stablefamilies}.} Let $\KK$ denote the set of quadratic fields with prescribed splitting types $\varsigma_p'$ given at a finite set $S$ of odd primes $p$. Let $(\Sigma_p)_p$ denote sets, where for each $p \notin S$, $\Sigma_p = \Sigmaall_p$ contains all possible splitting types, and for $p \in S$, $\Sigma_p = \{(\ast,\varsigma_p')\}$ consists of all possible splitting types compatible with $\varsigma_p'$. If we let $\Sigma_\infty^{\rm(a)} = \{((1111),(11)), \ ((112),(11)), \ ((22),(11))\}$, and $\Sigma_\infty^{\rm(b)} = \{((22),(2))\}$, we can define $\Sigma^{(\ast)}$ to be the collection $(\Sigma_p)_p$ and $\Sigma_\infty^{\rm(\ast)}$ for $\ast = {\rm a}$ or ${\rm b}$, which are both  acceptable collections. \phantom{$n_{\tau_{\varsigma_\infty}}^{\tau^{\varsigma^\infty}}$}

Recall that
$\NumC(\Sigma^{(\ast)};X,X^{\beta})$ counts the number of isomorphism classes of $D_4$-fields $L\in\LL(\Sigma^{(\ast)})$ such
that $|\CC(L)|<X$ and $|\D(L)|<X^{\beta}$.  As before, let $\complex(K)$ denote the number of pairs of complex embeddings of $K$.  From Theorem
\ref{thm:analyticmain2}, we have for  $\ast = {\rm a}$ or ${\rm b}$ and $\beta < 2/3$,
\begin{equation}\label{th2pfeq1}
  \NumC(\Sigma^{(\ast)};X,X^\beta)=
  \frac{X}{2\zeta(2)}\cdot\sum_{\substack{K\in\KK(\Sigma^{(\ast)})\\ |\Disc(K)|<
      X^\beta}}\frac{L(1,K/\mathbb{Q})}{
    L(2,K/\mathbb{Q})}\cdot\frac{2^{-\complex(K)}}{|\Disc(K)|} +
  o_\beta(X).
\end{equation}
On the other hand, we can also estimate $N_C(\Sigma^{(\ast)};X,X^\beta)$ using Theorem
\ref{thacceptablecount} as follows.  Consider the region
$$R_{X,\beta}:=\{(\D,\q)\in\R^2:|\D\cdot\q|<X,\;|\D|<X^{\beta}\}.$$ There exist
regions $R^{(\pm)}_{X,\beta}$, that are disjoint unions of
$\delta$-adic rectangles, such that $$R^{(-)}_{X,\beta}\subset
R_{X,\beta}\subset R^{(+)}_{X,\beta}$$ and such
that $$|\Vol(R_{X,\beta})-\Vol(R^{(\pm)}_{X,\beta})|\ll\delta\cdot
X\log (X^\beta).$$ The volume of $R_{X,\beta}$ is $X\log
X^{\beta}$. Therefore, from Theorem \ref{thacceptablecount}, we see that for $\beta < 1/2$,
\begin{equation}\label{th2pfeq2}
  \NumC(\Sigma^{(\ast)};X,X^\beta)=\frac{\zeta(2)}{2}\cdot X\log (X^\beta)\cdot \mu(\Sigma^{(\ast)}_\infty)\cdot\prod_{p}\mu(\Sigma_p)+o_\delta(X\log(X^\beta))+O(\delta X\log (X^\beta)).
\end{equation}
Equating the right hand sides of \eqref{th2pfeq1} and
\eqref{th2pfeq2}, dividing both sides by $X\log(X^{\beta})$, first letting $X^\beta$ tend to infinity, and then finally letting $\delta$ tend
to $0$, we obtain:
\begin{equation}\label{th2pfeqff}
  \frac{1}{2\zeta(2)}\cdot\sum_{\substack{K\in\KK(\Sigma^{(\ast)})\\ |\Disc(K)|<
      X}}\frac{L(1,K/\mathbb{Q})}{
    L(2,K/\mathbb{Q})}\cdot\frac{2^{-\complex(K)}}{|\Disc(K)|} \ \sim \ \frac{\zeta(2)}{2}\cdot\mu(\Sigma^{(\ast)}_\infty)\cdot\prod_p\mu(\Sigma_p)
  \cdot\log (X)
\end{equation}
It is easy to see that $2^{r_2(K)}\cdot\mu(\Sigma^{(\ast)}_\infty)$ is
always $1/2$, independent of $K\in\KK(\Sigma^{(\ast)})$. Furthermore,
the values of $\mu(\Sigma_p) = \sum_{(\varsigma_p,\varsigma_p') \in \Sigma_p} \mu(\maximalatp(\varsigma_p,\varsigma_p'))$ can be computed from Proposition
\ref{propdenmax}, and it then follows that the right hand sides of
Theorem \ref{stablefamilies}(a) and (b) are asymptotically equal to
\begin{equation*}
  \frac{\zeta(2)^2}{2}\cdot\prod_p\mu(\Sigma_p)\cdot\log (X).
\end{equation*}
  This concludes the proof of Theorem
\ref{stablefamilies}.  \hfill {$\Box$ \vspace{2 ex}}

\subsection{The proofs of Theorems \ref{MAINTHEOREM}
  and \ref{congruence conditions}}

We next obtain asymptotics for $\NNumQ^{(\delta)}(\Sigma;X,Y)$ when
$Y\gg X$. Recall that the outer automorphism $\outeraut$ of $D_4$
provides a non-isomorphic $D_4$-field $\outeraut(L)$ for each
$D_4$-field $L$, and the fields $L$ and $\outeraut(L)$ have the same
conductor but (possibly) different
invariants. Proposition~\ref{prop:discrelation} can be used to compute
the invariants of $\outeraut(L)$ in terms of the invariants of
$L$. Note that if $\D(L)>\q(L)$, then
$\D(\outeraut(L))<\q(\outeraut(L))$. Hence, for a collection of local
specifications $\Sigma$, we may relate counts of $D_4$-fields with $\D
> \q$ to counts of $D_4$-fields with $\D < \q$.

Given an acceptable collection $\Sigma$,
let $\outeraut(\LL(\Sigma))$ denote the family defined by
\begin{equation*}
\outeraut(\LL(\Sigma)):=\{\outeraut(L):L\in\LL(\Sigma)\}
\end{equation*}
It is clear that there exists another acceptable collection
$\phi(\Sigma)$ of local specifications such that $\phi(\LL(\Sigma)) =
\LL(\phi(\Sigma))$. Furthermore, for every odd prime $p$, Table 1 in
conjunction with Proposition \ref{propdenmax} and Lemma
\ref{lemdecden} for an acceptable collection $\Sigma$, we have
$\mu(\Sigma_p)=\mu(\phi(\Sigma_p))$. An acceptable family $\Sigma$ is
said to be {\it very stable at $2$} if the set $\Sigma_2$ either contains
all splitting types with central inertia (pairs
$(\varsigma_2(L_1),\varsigma_2(K_1))$ in the latter two groups of
Table~1) or it contains none of them. Our next result computes the
number of $D_4$-fields in $\LL(\Sigma)$ satisfying $X \leq |\q(L)|
\leq (1+\delta)X$ and $Y \leq |\D(L)| < (1+\delta)Y$ when $Y$ is much
larger than $X$.

\begin{thm}\label{thlargedcount}
  Let $\Sigma$ be an acceptable collection of local specifications
  that is very stable at $2$. Let $X$ and $Y$ be positive real numbers such
  that $X(\log X)^2=o(Y)$. Then we have
  \begin{equation*}
    \NNumQ^{(\delta)}(\Sigma;X,Y)=
    \frac{\zeta(2)}{2}\cdot\delta^2\cdot\mu(\Sigma_\infty)\cdot\prod_p\mu(\Sigma_p)\cdot XY+o(XY).
  \end{equation*}
\end{thm}

\begin{proof}
For a prime $p$ and an integer $a\geq 1$, let $\VV(p,a)$ denote the
set of $D_4$-fields $L$ such that $J_p(L) = p^a$. (Recall that $J(L)$
is defined in \eqref{ci}, and an odd prime $p\mid J(L)$ if only if $L$
has splitting type $((2^2),(2))$, or $((1^21^2),(11))$.)  Let $\VV_p$
denote the union of $\VV(p,a)$ over all $a\geq 1$, and for any integer
$n\geq 1$, let $\LL(\Sigma)^{(n)}$ denote the set of fields $L$ in
$\LL(\Sigma)$ such that $J(L)=n$. One can check that
$\LL(\Sigma)^{(n)}$ is defined by an acceptable collection $\Sigma^{(n)}$ of local
specifications that is very stable at 2, i.e., $\LL(\Sigma^{(n)}) = \LL(\Sigma)^{(n)}$. We have
\begin{equation*}
  \begin{array}{rcl}
    \NNumQ^{(\delta)}(\Sigma;X,Y)&=&
    \displaystyle\sum_{n\geq 1}\NNumQ^{(\delta)}(\Sigma^{(n)};X,Y)\\[.3in]
    &=&\displaystyle\sum_{n\geq 1}
    \NNumQ^{(\delta)}(\phi(\Sigma^{(n)});Yn,X/n).
  \end{array}
\end{equation*}
For a fixed integer $M$, we use Theorem \ref{thacceptablecount} to
evaluate $\NNumQ^{(\delta)}(\phi(\Sigma)^{(n)};Yn,X/n)$ for $n\leq M$
and Proposition \ref{propunifJ} to bound
$\NNumQ^{(\delta)}(\phi(\Sigma)^{(n)};Yn,X/n)$ for $n>M$. Altogether,
we obtain
\begin{equation*}
  \NNumQ^{(\delta)}(\Sigma;X,Y)\sim
  \frac{\zeta(2)}{2}\cdot\delta^2\cdot\mu(\phi(\Sigma)_\infty)
  \cdot\Bigl(\sum_{n= 1}^M\prod_{p^a\parallel n}\mu(\phi(\Sigma)_p\cap\VV(p,a))
  \cdot\prod_{p\nmid n}\mu(\phi(\Sigma)_p\backslash\VV_p)\Bigr)\cdot XY
\end{equation*}
up to an error of
$o_{\delta,M}(XY)+O_{\epsilon,\delta}(XY/M^{1-\epsilon})$, where we
assume that $a\geq 1$.  Dividing by $XY$, letting $X$ and $Y$ tend to
infinity, and then letting $M$ tend to infinity, we obtain
\begin{equation*}
  \begin{array}{rcl}
    \displaystyle\lim_{M\rightarrow \infty}\lim_{X,Y\to\infty}
    \frac{\NNumQ^{(\delta)}(\Sigma;X,Y)}{\delta^2XY}&=&
  \displaystyle\frac{\zeta(2)}{2}\cdot\mu(\phi(\Sigma)_\infty)
  \cdot\sum_{n\geq 1}\Bigl(\prod_{p^a\parallel n}\mu(\phi(\Sigma)_p\cap\VV(p,a))
  \cdot\prod_{p\nmid n}\mu(\phi(\Sigma)_p\backslash\VV_p)\Bigr)\\[.3in]&=&
  \displaystyle\frac{\zeta(2)}{2}\cdot\mu(\phi(\Sigma)_\infty)\cdot
  \sum_{n\geq 1}\Bigl(\prod_{p}\mu(\phi(\Sigma)_p\backslash\VV_p)
  \cdot\prod_{p^a\parallel n}\frac{\mu(\phi(\Sigma)_p\cap\VV(p,a))}
  {\mu(\phi(\Sigma)_p\backslash\VV_p)}
  \Bigr)\\[.3in]&=&
  \displaystyle\frac{\zeta(2)}{2}\cdot\mu(\phi(\Sigma)_\infty)
  \cdot\prod_{p}\mu(\phi(\Sigma)_p\backslash\VV_p)\cdot
  \prod_p\Bigl(1+\sum_{a\geq 1}\frac{\mu(\phi(\Sigma)_p\cap\VV(p,a))}
  {\mu(\phi(\Sigma)_p\backslash\VV_p)}\Bigr)\\[.3in]&=&
  \displaystyle\frac{\zeta(2)}{2}\cdot\mu(\phi(\Sigma)_\infty)\cdot
  \prod_p\mu(\phi(\Sigma)_p).
  \end{array}
\end{equation*}
Since $\mu(\Sigma_p)=\mu(\phi(\Sigma)_p)$ for all primes $p$ and
$\mu(\Sigma_\infty) = \mu(\phi(\Sigma_\infty))$, we obtain the result.
\end{proof}

We now have theorems computing $\NNumQ^{(\delta)}(\Sigma; X, Y)$ when
$X(\log X)^2 = o(Y)$ (Theorem~\ref{thacceptablecount}) and when $Y(\log Y)^2 = o(X)$ (Theorem~\ref{thlargedcount} with identical
right hand sides. Our last task in proving Theorem~\ref{congruence
  conditions} is to show that the region that neither Theorem
\ref{thacceptablecount} nor \ref{thlargedcount} covers contributes negligibly
to $\NNumQ^{(\delta)}(\Sigma; X, Y)$. For that we need the following
lemma.
\begin{lemma}\label{lemmediumrange}
The number of $D_4$-fields $L$ such that $|\CC(L)|\leq X$ and $X(\log
X)^{-3}\leq|\D(L)|\leq X(\log X)^3$ is bounded by $O(X\log\log X)$.
\end{lemma}
\begin{proof}
The number of $D_4$-fields satisfying the conditions of the lemma can
be estimated as a sum of ratios of $L$-values from
Theorem~\ref{thm:analyticmain2}. This sum can be bounded, using
Proposition \ref{thm:analyticmain}, by
  \begin{equation*}
X\cdot \sum_{D=X(\log X)^{-3}}^{X(\log X)^{3}}\frac{1}{D},
  \end{equation*}
  yielding the lemma.
\end{proof}

\vspace{.1in}
\noindent{\it Proof of Theorem \ref{congruence conditions}.}  The
invariants $\D$ and $\q$ of a $D_4$ field with absolute conductor bounded by
$X$ satisfy $|\D\cdot\q|<X$. Consider the region
$R_X:=\{(\D,\q)\in\R^2:|\D\cdot\q|<X\}$.  We bound the number of $D_4$-fields
$L$ with $X(\log X)^{-3}\leq\D(L)\leq X(\log X)^3$ using Lemma
\ref{lemmediumrange}, and estimate the rest of the $D_4$-fields using
Theorems \ref{thacceptablecount} and \ref{thlargedcount} with an argument
identical to the proof of Theorem \ref{stablefamilies}, obtaining
\begin{equation*}
\NumD(\Sigma,X)\sim
\frac{\zeta(2)}{2}\cdot\mu(\Sigma_\infty)\cdot\prod_p\mu(\Sigma_p)\cdot X\log (X).
\end{equation*}
Let $\varsigma_p(L_p,K_p)$ denote the splitting type of a pair
$(L_p,K_p)$ of local extensions of $\Q_p$. From Propositions
\ref{prop:expectednum} and \ref{propdenmax}, we obtain for each prime $p$,
\begin{equation*}
\begin{array}{rcl}
  \displaystyle\frac{1}{\#\Aut(L_p,K_p)}
  \cdot\frac{1}{\CC_p(L_p)}\cdot \Bigl(1-\frac1{p}\Bigr)^2&=&
\displaystyle\frac{\mu(\maximalatp(\varsigma_p(L_p,K_p)))}{1-p^{-2}},
\end{array}
\end{equation*}
and at $\infty$, $\#\Aut(L_\infty,K_\infty) = \tau_{\varsigma(L_\infty),\varsigma(K_\infty)}$ (see Theorem \ref{thmaincount} for the definition of $\tau_{\varsigma}$). This concludes the proof of Theorem \ref{congruence conditions}.
\hfill {$\Box$ \vspace{2 ex}}

Theorem \ref{MAINTHEOREM} follows directly from Theorem \ref{congruence
  conditions} in conjunction with the density computations in
Theorem~\ref{thdenfull}.

\subsection{The proof of Theorem \ref{classgroups}}
We end this article with the proof and a discussion of Theorem
\ref{classgroups}.

  Let $K\in\KK$ be a quadratic field, and for each $p \in S$, let $\varsigma$ denote the prescribed splitting type at $p$ for $\KK$. Since finite abelian groups are
  isomorphic to their duals, we see that $\#\Cl(K)[4]-\#\Cl(K)[2]$,
  the number of elements in $\Cl(K)$ having exact order $4$, is equal
  to twice the number of index-$4$ subgroups of $\Cl(K)$ whose quotients are
  cyclic.  By class field theory, such index-$4$ subgroups of $\Cl(K)$
  are in bijection with isomorphism classes of unramified extensions $M$ of $K$ with
  $\Gal(M/K)=C_4$. Such an extension $M$ is Galois over $\Q$ with
  Galois group $D_4$. Conversely, if $M$ is an octic $D_4$-field whose
  splitting type at every prime $p$ lies in the first two quadrants of
  Table 1, then $M$ is unramified over $K$, its quadratic subfield
  fixed by $C_4\subset D_4$. Furthermore, it is easily checked from
  Table 1 that $M$ and $K$ have the same conductor.

  Now we define three collections of local specifications $\Sigma^{(i)}$
  corresponding to the three cases in Theorem \ref{classgroups}. First, let $\KK^{(\ast)}$ be the
  subset of $K \in \KK$ with $\varsigma_\infty(K) = (11)$ when $\ast =
  {\rm a}$ or ${\rm c}$, and $\varsigma_\infty(K) = (2)$ when $\ast =
  {\rm b}$, and define
  \begin{equation*}
\Sigma_\infty^{(\ast)} := \begin{cases}
\{((1111),(11))\} & \mbox{ if }\ \ast = {\rm a}, \\
\{((112),(11)),\,((22),(2))\} & \mbox{ if } \  \ast={\rm b},\\
\{((1111),(11)),\,((22),(11))\} & \mbox{ if } \ \ast = {\rm c}.
\end{cases}
\end{equation*}
Next we define $\Sigma_p$ for all finite primes $p$. If $p\not\in S$,
  define
\begin{equation*}
  \Sigma_p=\left\{ (\varsigma_p, \varsigma_p') \ : \ (\varsigma_p, \varsigma_p')
  \text{ lacks central inertia} \right\}.
\end{equation*}
For $p\in S$, define
\begin{equation*}
\displaystyle \Sigma_p=\begin{cases}
\{((1111),(11)),\,((22),(11)),\,((4),(2))\} & \mbox{ if }\varsigma=(11),\\
\{((112),(11)),\,((22),(2))\} & \mbox{ if }\varsigma=(2),\\
\{((1^211),(11)),\,((1^22),(11)),\,((1^21^2),(1^2)),\,((2^2),(1^2))\}
&\mbox{ if }\varsigma=(1^2);
\end{cases}
\end{equation*}

  Let $\Sigma^{(\ast)}$ then denote collection of local specifications consisting of $(\Sigma_p)_p$ along with $\Sigma_{\infty}^{(\ast)}$ for $\ast = {\rm a}$, ${\rm b}$, or ${\rm c}$, which is acceptable and very stable at 2. It is easily checked from Table 1 that the Galois closures of fields
  in $\LL(\Sigma^{(\ast)})$ correspond to cyclic quartic unramified extensions of fields in
  $\KK^{(\ast)}$. When $\ast={\rm a}$ or ${\rm b}$, $\LL(\Sigma^{(\ast)})$ corresponds to order 4 elements in the class groups of $\KK^{(\ast)}$. On the other hand, $\LL(\Sigma^{({\rm c})})$
  corresponds to cyclic quartic extensions of $\KK^{({\rm c})}$ that are unramified at every finite place, but
  possibly ramified at infinity; thus, they correspond to order 4 elements in the narrow class groups of real quadratic fields in $\KK$. Furthermore, from \eqref{d4diagram}, it follows that exactly two distinct isomorphism classes of 
  $D_4$-fields yield the same Galois closure, but additionally every index-4 subgroup of the class group corresponds to two order 4 ideal classes. Thus, we conclude that the left hand sides of  Theorem \ref{classgroups}$(\ast)$ are equal to $
  N_{D_4}(\Sigma^{(\ast)},X)$, for $\ast={\rm a}$, ${\rm b}$, and
  ${\rm c}$. The theorem follows from Theorem \ref{congruence
    conditions} along with density computations following from Proposition \ref{propdenmax}. \hfill {$\Box$ \vspace{2 ex}} \vspace{-10pt}
\begin{rmk}\label{FK} Fouvry-Kl\"uners \cite{fouvrykluners} prove that the average size of
$2^{\rk_4(\Cl^+(K))}$ over imaginary (respectively, real) quadratic fields $K$
ordered by discriminant is equal to $2$ (respectively, $\frac{3}{2}$), where $\rk_4(\Cl^+(K)) =
\dim_{\bF_2}(\Cl^+(K)^4/\Cl^+(K)^2)$. Thus, we have
	$$\#\Cl^+_4(K) - \#\Cl^+_2(K) =
(2^{\rk_4(\Cl^+(K))}-1)\cdot(\#\Cl^+_2(K)).$$ It is well-known from
genus theory that for any quadratic field $K$, $\#\Cl^+_2(K) =
2^{\omega(D)-1}$, where $\omega(D)$ is equal to the number of prime
factors of $D$. Additionally, the asymptotic in
Theorem \ref{classgroups}(b) (respectively, Theorem \ref{classgroups}(c)) is
also true if one replaces the summand on the left hand side with
simply $\#\Cl_2(K)$ (respectively, $\#\Cl^+_2(K)/2$). Thus, the
average value of the product $(2^{\rk_4(\Cl^+(K))}-1)\cdot(\#\Cl^+_2(K))$
is equal to the product of the average value of $(2^{\rk_4(\Cl^+(K))}-1)$
and the average size of $\Cl^+_2(K)$. Similar analysis holds
for class groups of real quadratic fields. 
\end{rmk}

\subsubsection*{Acknowledgements}
We thank Manjul Bhargava, Henri Cohen, Dorian Goldfeld, Joseph Gunther, John W.~Jones, J\"urgen Kl\"uners, Samuel Ruth, Xiaoheng Wang, and Melanie
Matchett Wood for helpful conversations and comments.

\bibliography{final}{}
\bibliographystyle{plain}

\end{document}